\documentclass[11pt]{article}
\usepackage{amssymb}
\usepackage{amsmath}
\usepackage{latexsym}
\usepackage{hyperref}
\usepackage{comment}

\usepackage{bm}

\usepackage{amsmath, amsthm, amssymb}
\usepackage{mathtools}
\usepackage{enumerate}
\usepackage{tikz}
\usepackage{commath}
\usepackage{breqn}
\usepackage{graphicx}
\usepackage{cancel}
\usepackage{bm} 

\usepackage[all]{xy}
\usepackage{color}
\usepackage[normalem]{ulem}

\numberwithin{equation}{section}
\newtheorem{thm}[equation]{Theorem}

\newtheorem{defn}[equation]{Definition}
\newtheorem{rem}[equation]{Remark}
\newtheorem{lem}[equation]{Lemma}

\def \f{\Phi}

\def \tf{\tilde{\f}}
\def \tfi{\tilde{\f}^{(-1)}}

\def \tu{\tilde{u}}
\def \te{\tilde{\varphi}}
\def \tv{\tilde{v}}
\def \tw{\tilde{w}}

\def \w{w}

\def \optr{\Upsilon_{\f}}
\def \optrt{\Upsilon_{\tf}}
\def \optri{\optr^{-1}}
\def \optrit{\optrt^{-1}}

\def \into{\int_{\Omega}}

\def \inttpo{\int_{\tf(\Omega)}}
\def \inttpos{\int_{\partial \tf(\Omega)}}

\newcommand{\cu}{{\operatorname{curl}}}
\newcommand{\di}{{\rm div}}
\newcommand{\gr}{{\rm grad}}
\newcommand{\dig}{{\rm div}_{\scriptscriptstyle\Gamma}}
\newcommand{\grg}{{\rm grad}_{\scriptscriptstyle\Gamma}}

\newcommand{\cau}{{\mathcal{C}}}

\title{Shape sensitivity analysis for electromagnetic cavities }
\author{Pier Domenico Lamberti and Michele Zaccaron}

\date{\today }

\begin{document}

\newcommand{\N}{\mathbb{N}}
\newcommand{\R}{\mathbb{R}}

\newcommand{\xn}{X_{\rm \scriptscriptstyle\tiny N}}

\newcommand{\Hmm}[1]{\leavevmode{\marginpar{\tiny%
$\hbox to 0mm{\hspace*{-0.5mm}$\leftarrow$\hss}%
\vcenter{\vrule depth 0.1mm height 0.1mm width \the\marginparwidth}%
\hbox to
0mm{\hss$\rightarrow$\hspace*{-0.5mm}}$\\\relax\raggedright #1}}}


\maketitle
\begin{center}    

\end{center}
\vspace{0.4cm}

\noindent
{\bf Abstract:} We study the dependence of the eigenvalues of time-harmonic  Maxwell's equations in a cavity upon variation of  its shape. 
The analysis concerns all eigenvalues both simple and multiple. We provide 
analyticity results for the dependence of the elementary symmetric functions of the eigenvalues splitting a multiple eigenvalue, as well as a Rellich-Nagy-type result describing the corresponding  bifurcation phenomenon. We also address an isoperimetric problem and characterize the critical cavities for the symmetric functions of the eigenvalues subject to isovolumetric or isoperimetric  domain perturbations and prove that balls are critical. 
We include known formulas for the eigenpairs in a ball and calculate the first one. 

\vspace{11pt}


\vspace{6pt}

\section{Introduction}
This paper is devoted to the analysis  of the dependence of  eigenvalues of time-harmonic Maxwell's equations in a cavity 
$\Omega $ of ${\mathbb{R}}^3$  upon variation of the shape of $\Omega$. Here cavities are  understood as bounded connected open sets, in other words, bounded domains of the Euclidean space. Moreover, unless otherwise indicated, they are assumed to be sufficiently smooth  in order to guarantee the validity of 
a number of facts, primarily the celebrated Gaffney inequality. 

We point out  from the very beginning that   the study of electromagnetic cavities has many applications, for example in designing cavity resonators or shielding structures for electronic circuits,   see e.g.,  \cite[Chp.~10]{hanyak} for a detailed  introduction to this field of investigation. See 
also  the classical books \cite{ces}, \cite{dali3}, \cite{gira} and the more recent \cite{monk, ned, rsy} for extensive discussions and references concerning the mathematical theory of electromagnetism.  We also refer to  the  well-known  papers \cite{ co91, cocoercive, coda} by M. Costabel and M. Dauge, as well as to the more recent  papers \cite{ammari, bauer, calamo, cocomo, coda2, lamstra, pauly, yin}.

Recall that  time-harmonic Maxwell's equations in a homogeneous isotropic medium filling a domain $\Omega$ in $\R^3$   can be written  as 
\begin{equation}\label{Maxwelltimeharm}
\cu E - {\rm i} \, \omega  \, H = 0\,,\,\,\cu H + {\rm i} \, \omega  \, E = 0\,,
\end{equation}
where $E, H$ denote the spatial parts of the electric and the magnetic field respectively and  $\omega  > 0$ is the angular frequency.  
Here, for simplicity  the electric permittivity $\varepsilon$  and the  magnetic permeability   $\mu$  of the medium have been normalized by setting $\epsilon =\mu =1$.
Note that the solutions $E$, $H$ to \eqref{Maxwelltimeharm} are  divergence-free.  

As far as the boundary conditions are concerned, we consider those of a perfect conductor, hence we require that  
\begin{equation}\label{bdc}
\nu \times E=0 \ \ {\rm and }\ \  H\cdot \nu =0,
\end{equation}
where $\nu $ denotes the unit outer normal to the boundary  $\partial \Omega$ of $\Omega$.

Operating by $\cu$ in each of the equations 
\eqref{Maxwelltimeharm} and setting $\lambda =\omega^2$, one obtains the following well-known boundary value problems 
\begin{equation}\label{claE}\left\{
\begin{array}{ll}
\cu^2 E=\lambda E,&\ {\rm in}\ \Omega,\\
\nu \times E=0,&\ {\rm in }\ \partial \Omega,
\end{array}\right.
\end{equation}
and 
\begin{equation}\label{claH}\left\{
\begin{array}{ll}
\cu^2 H=\lambda H,&\ {\rm in}\ \Omega,\\
 H\cdot \nu =0,&\ {\rm in }\ \partial \Omega,\\
\nu \times \cu H=0,&\ {\rm in }\ \partial \Omega .\\
\end{array}\right.
\end{equation}
Note that the solutions to \eqref{claE} automatically satisfy the boundary condition $\cu E \cdot \nu =0,\ {\rm on }\ \partial \Omega$, see Lemma~\ref{fundamental}.
As it is natural, each of the two problems provide the same eigenvalues (cf., \cite{zhang}) which can be represented by an increasing, divergent sequence
$$
0<\lambda_1[\Omega]\le \lambda_2[\Omega]\le \dots \le \lambda_n[\Omega]\le \dots 
$$
where each eigenvalue is repeated according to its multiplicity (that is, the dimension of the corresponding eigenspace), which is finite. Note that, in fact, the corresponding resolvent operators are compact and self-adjoint.   

In this paper, we study the dependence of $\lambda_n[\Omega]$ on $\Omega$ and we aim  at proving analyticity results for all eigenvalues, both simple and multiple, and addressing an optimization problem concerning the role of balls in isovolumetric or isoperimetric domain perturbations.  

Despite the importance of problem \eqref{Maxwelltimeharm}, these issues are not much discussed in the literature and we are aware only of the paper \cite{jimbo} by  S. Jimbo which provides a Hadamard-type formula for the shape derivatives of simple eigenvalues and quotes 
the book  \cite{hira} by  K. Hirakawa  where an analogous but different  formula is provided  on the base of heuristic computations.  In particular, we note that the formula in \cite[(4-88), p. 92] {hira} agrees with our Hadamard formula  \eqref{hadamardsimple} below.  

On the contrary, analogous stability and optimization issues for problems arising in elasticity theory have been largely investigated in the literature and many results are available, see e.g., the classical  monograph  \cite{henry} by D. Henry.  See also \cite{bubu, henrot}  for  more information about the  variational approach, and the forthcoming monograph \cite{dalamu} for a functional analytic approach to stability problems.    In particular, it appears that for  classical boundary value problems involving rotational invariant operators, a kind of principle holds, namely,  {\it  all simple eigenvalues and the elementary symmetric functions of multiple eigenvalues depend real analytically  on the domain, and balls are corresponding critical domains with respect to isovolumetric and isoperimetric domain perturbations}, see  \cite{bula2015}.
In this paper, we prove that the eigenvalue problem \eqref{Maxwelltimeharm} obeys this principle. In particular, in Theorem~\ref{symmetrican}  we provide the analyticity result with the appropriate Hadamard formula,     and in Theorem~\ref{over}
 we prove that an open set $\Omega$ is a critical domain, under the volume  constraint ${\rm Vol }(\Omega)={\rm constant}$, for the elementary symmetric  functions of the eigenvalues bifurcating from an eigenvalue $\lambda$ of multiplicity $m$
if and only if the following extra boundary  condition is satisfied:
\begin{equation}\label{overintro}
\sum_{k=1}^m\left(|E^{(i)}|^2-|H^{(i)}|^2\right)= c\ \ {\rm on}\ \partial \Omega , 
\end{equation}
where $c$ is a constant,  $E^{(i)}$, $i=1, \dots m$, is an orthonormal basis in $(L^2(\Omega))^3$ of the (electric) eigenspace associated with $\lambda$, and $H^{(i)}=  -{\rm i}  \operatorname{curl}    E^{(i)}/ \sqrt{\lambda} $ is the magnetic field associated with $E^{(i)}$ as in \eqref{Maxwelltimeharm}. Then, in Theorem~\ref{balls} we prove that condition
\eqref{overintro} is satisfied if $\Omega$ is a ball. Similarly, in the case of isoperimetric domain perturbations and perimeter constraint  ${\rm Per }(\Omega)={\rm constant}$, the extra boundary  condition reads
\begin{equation}\label{overper}
\sum_{k=1}^m\left(|E^{(i)}|^2-|H^{(i)}|^2\right)= c {\mathcal{H}}\ \ {\rm on}\ \partial \Omega , 
\end{equation}
where $c$ is a constant and ${\mathcal{H}}$ is the mean curvature of $\partial \Omega$ (the sum of the principal curvatures). 

It would be interesting to characterise all domains for which the above conditions are satisfied.

We note that using the elementary symmetric functions
of multiple eigenvalue, rather than the eigenvalues themselves, is quite natural since domain perturbations typically split the multiplicities of 
the eigenvalues and produce bifurcation phenomena responsible for corner points in the corresponding  diagrams. Moreover, we believe that in the specific case of Maxwell equations, being able to deal with multiple eigenvalues is quite important  since all eigenvalues of the Maxwell system in a ball are multiple (in particular, the first eigenvalue  has multiplicity equal to three, see Theorem~\ref{firsteigen}).  

In order to discuss the behaviour of the eigenvalues $\lambda_n[\Omega]$, it is clearly equivalent to study either problem \eqref{claE} or problem \eqref{claH}.  We have decided to choose  the first one.  

Our method is based on the theoretical results obtained in \cite{lala} which concerns the case of general families of compact self-adjoint operators in Hilbert space with variable scalar product. Here we consider a class of open sets $\Phi (\Omega)$ identified by a class of diffeomorphisms $\Phi$ defined on a fixed reference domain $\Omega$. Then our problem is set on $\Phi (\Omega)$ and pulled-back to $\Omega$ by means of the covariant Piola transform associated with $\Phi$. This allows to recast the problem on $\Omega$ and to reduce it to the study of a curl-div problem with parameters depending on $\Phi$, the eigenvalues of which are exactly $\lambda_n[\Phi (\Omega)]$. Then, passing to the analysis of the corresponding  resolvents defined in $(L^2(\Omega ))^3$ (with an appropriate equivalent scalar product depending on $\Phi $) we can prove our analyticity results for the maps $\Phi \mapsto \lambda_n[\Phi (\Omega)] $
and study their critical points under volume constraint ${\rm Vol}\, ( \Phi (\Omega))={\rm const.}$ or perimeter constraint  $|{\rm Per}\,  \Phi (\Omega)|={\rm const.}$

We note that the families of compact self-adjoint operators under consideration are obtained by following the method of \cite{cocoercive} which consists in adding the  penalty term $-\tau  \nabla \operatorname{div} E$    in the equation \eqref{claE}, depending on an arbitrary positive number 
$\tau$. Then it is enough to observe that the eigenvalues of the penalized   problem are given  by the union of the eigenvalues of problem \eqref{claE} and the eigenvalues of the Dirichlet Laplacian $-\Delta $  in $\Omega$, multiplied by $\tau $, see \cite[Theorem~1.1]{coda}, Lemma~\ref{union}  and Remark~\ref{risonanzarem}. In particular, it follows that  the analyticity result stated in the first part of our Theorem~\ref{symmetrican}  below yields (in the case of regular domains) also the analyticity result proved in \cite{lala} for the eigenvalues of the Dirichlet Laplacian. 

Besides the results  described above, we would like to highlight two by-pass products of our analysis. First, we prove a Rellich-Nagy-type result describing the bifurcation phenomenon mentioned above. Namely, given an eigenvalue $\lambda $ of multiplicity $m$, say $\lambda =\lambda_n=\dots =\lambda_{n+m-1}$, and a perturbation of $\Omega$ of the form $\Phi_{\epsilon }(\Omega )$  with $\Phi_{\epsilon } (x) = x+ \epsilon V(x) $ for all $x\in \Omega$ where $V$ is a ${\mathcal{C}}^{1,1}$ vector field defined on $\overline{ \Omega }$,  we  prove that the set of right  derivatives at $\epsilon =0$ of $\lambda_{n+k}[\Phi_{\epsilon} (\Omega)] $ for all $k=0, \dots , m-1$ (which coincides with the set of left derivatives, although each right and left derivative may be different) 
are given by the eigenvalues of the matrix  $(M_{i,j})_{i,j=1,\dots ,m}$ where
\begin{eqnarray}M_{i,j} =
 \int_{\partial \Omega} \left(   \lambda  E^{(i)}\cdot  E^{(j)}  - \operatorname{curl}  E^{(i)}\cdot  \operatorname{curl} E^{(j)}\right)  V \cdot \nu  \, d\sigma \, ,
\end{eqnarray}
 and $E^{(i)}$, $i=1, \dots m$, is a (real) orthonormal basis in $(L^2(\Omega))^3$ of the (electric) eigenspace associated with $\lambda$. In particular, if $\lambda_n$ is a simple eigenvalue we get the Hadamard formula
 \begin{equation}
 \label{hadamardsimple}
 \left. \frac{d}{d\epsilon }\lambda _{n}[\Phi _{\epsilon} (\Omega )]\right\vert_{\epsilon=0}=\int_{\partial \Omega} \left(   \lambda|  E|^2  - |\operatorname{curl}  E|^2 \right)  V \cdot \nu  \, d\sigma = \lambda \int_{\partial \Omega} \left(     |E|^2  - |H|^2 \right)  V \cdot \nu  \, d\sigma \, ,
 \end{equation}
 where $E$ is an eigenvector normalized in $(L^2(\Omega))^3$ and  $H=-{\rm i}  \operatorname{curl}    E/ \sqrt{\lambda}$ as above.

Second, by using these formulas we prove a Rellich-Pohozaev formula for the   Maxwell eigenvalues, namely any eigenvalue $\lambda$ can be  represented by the formula 
\begin{eqnarray}\lambda =
\int_{\partial \Omega} \left( | \operatorname{curl}    E |^2   - |\operatorname{curl} H|^2 \right)  x \cdot \nu  \, d\sigma \, ,
\end{eqnarray}
where $E$ is any (electric) eigenvector associated with $\lambda$ normalized in  $(L^2(\Omega))^3$  and  $H=-{\rm i}  \operatorname{curl}    E/ \sqrt{\lambda}$. In particular, we have the identity
\begin{eqnarray}
\int_{\partial \Omega} \left( |   H |^2   - |E|^2 \right)  x \cdot \nu  \, d\sigma =1 .
\end{eqnarray}

We note that our formulas are proved under the assumption that the eigenvectors under consideration are of class $H^2$ (which means that  they have  square summable derivatives up to the second order)   and that this assumption is satisfied if the corresponding domain  is sufficiently regular, for example of class $C^{2,1}$, see Remark~\ref{regularity}. 

This paper is organized as follows. Section~\ref{preliminaries} is devoted to a number of preliminary results on the function spaces involved and to the variational formulations of the eigenvalue problems under consideration. In Section~\ref{transec} we formulate the domain perturbation problem and the corresponding domain transplantation process. In Section~\ref{analitic} we prove our main analyticity results, while in Section~\ref{oversec}   we address the optimization problem, we characterize the critical domains by means of appropriate overdetermined problems and prove that balls are critical domains. In Appendix, for the convenience of the reader, we include a few known facts about the eigenvalue problem in a ball and we compute the first eigenpair.


\section{Preliminaries on the  eigenvalue problem}
\label{preliminaries}

In this paper, the vectors of $\mathbb{R}^3$ are understood  as row vectors.  The transpose of a matrix  $A$ is denoted by $A^T$, hence  if $a \in \mathbb{R}^3$, then $a^T$ is a column vector.  If $a,b \in \mathbb{R}^3$ are two vectors, we denote by $\cdot$ the usual scalar product, that is $a \cdot b = a b^T$.

Let $\Omega$ be a bounded open set  in $\R^3$.  Since  problems \eqref{claE} and \eqref{claH} are associated with self-adjoint operators, for the sake of simplicity and without any loss of generality, the space  $L^2(\Omega)$ is understood  as a space of real-valued functions. 
In particular, the usual scalar product of two vector fields  $u,v\in (L^2(\Omega))^3$  is given by 
$\int_{\Omega}u\cdot vdx=\int_{\Omega}uv^Tdx $.

We  denote by  $H(\cu, \Omega)$  the space of vector fields
$u\in  (L^2(\Omega))^3 $ with distributional curl in  $(L^2(\Omega))^3$, endowed with the norm
$$ ||u||_{H(\cu, \Omega)} = \left( ||u||^{2}_{ L^2(\Omega)} + ||\cu u||^{2}_{L^2(\Omega)} \right)^{1/2}. $$
We denote by  $H_{0}(\cu, \Omega)$ the closure in $H(\cu, \Omega)$ of the space of ${\mathcal{C}}^{\infty}$-functions with compact support
in $\Omega$.  Similarly, we denote by $ H(\di, \Omega) $ the space of vector fields
$u\in  (L^2(\Omega))^3 $ with distributional divergence in  $(L^2(\Omega))^3$, endowed with the norm
$$ ||u||_{H(\di, \Omega)} = \left( ||u||^{2}_{ L^2(\Omega)} + ||\di u||^{2}_{L^2(\Omega)} \right)^{1/2}. $$
Moreover, we consider the space
$$
 X_{\rm \scriptscriptstyle\tiny N}(\Omega) = H_0(\cu, \Omega) \cap H (\di, \Omega)
$$
and we endow it with the norm defined by 
$$
 ||u||_{  X_{\rm \scriptscriptstyle\tiny N}(\Omega)  }\! =\! ( ||u||^{2}_{L^2(\Omega) } + ||\cu u||^{2}_{L^2(\Omega)} + ||\di u||^{2}_{L^2(\Omega)} )^{1/2} ,
$$
for all $u\in  X_{\rm \scriptscriptstyle\tiny N}(\Omega) $. Finally, we set 
$$X_{\rm \scriptscriptstyle\tiny N}(\di \, 0, \Omega) = \{u \in X_{\rm \scriptscriptstyle\tiny N}(\Omega) : \di u = 0 \,\, {\rm in}\ \Omega \}.$$

For details on  these operators and  spaces we refer to  \cite{ces}, \cite{dali3}, \cite{gira},  \cite{rsy}.
\newline

If $\Omega$ is  sufficiently regular, say of class ${\mathcal{C}}^{1,1}$, the space $\xn (\Omega) $ is continuously embedded into the space $(H^1(\Omega))^3$ of vector fields with components in the standard Sobolev 
space $H^1(\Omega)$ of functions in $L^2(\Omega)$ with first order weak derivatives in $L^2(\Omega)$. On the other hand, since $H^1(\Omega)$ is compactly embedded into $L^2(\Omega)$, it follows that also $\xn (\Omega) $ is compactly embedded into  $(L^2(\Omega))^3$. We note that the compactness of this embedding 
holds also under weaker assumptions on the regularity of $\Omega$. 
More precisely, we have the the following theorem the proof of which can be found in  \cite[Lemma~3.4, Theorem~3.7]{gira}.

\begin{thm}\label{gaffney} The following statements hold.
\begin{itemize}
\item[(i)] If  $\Omega $ is a bounded, simply connected open set in $\R^3$ of class ${\mathcal{C}}^{0,1}$ and $\partial \Omega$ has only one connected component then 
 there exists $c>0$ such that
$$\| u\|_{L^2(\Omega)}\le c \, \|  \cu u\|_{L^2(\Omega)},$$
for all $u\in X_{\rm \scriptscriptstyle\tiny N}(\di \, 0, \Omega)$, and
$$\| u\|_{L^2(\Omega)}\le c \, \left( \|  \cu u\|_{L^2(\Omega)} +  \|  \di u\|_{L^2(\Omega)} \right),$$
for all  $u\in X_{\rm \scriptscriptstyle\tiny N}( \Omega)$. Moreover, the embedding $X_{\rm \scriptscriptstyle\tiny N}( \Omega)\subset  (L^2(\Omega))^3$ is compact. 
\item[(ii)] If $\Omega$ is  a bounded  open set in $\R^3$ of class ${\mathcal{C}}^{1,1}$ then $ X_{\rm \scriptscriptstyle\tiny N}( \Omega)$ is continuously embedded into
$(H^1(\Omega))^3$,  and there exists
$c>0$ such that the Gaffney inequality 
\begin{equation}\| u\|_{(H^1(\Omega))^3}\le c \, \left( \| u\|_{L^2(\Omega)}+   \|  \cu u\|_{L^2(\Omega)} +  \|  \di u\|_{L^2(\Omega)} \right),\end{equation}
holds for all $u\in   X_{\rm \scriptscriptstyle\tiny N}( \Omega)$. In particular, the embedding $ X_{\rm \scriptscriptstyle\tiny N}( \Omega))\subset  (L^2(\Omega))^3$ is compact.
\end{itemize}
\end{thm}

At some point, we shall also need the following 

\begin{lem}\label{fundamental}
Let $\Omega$ be a bounded open set in ${\mathbb{R}}^3$ of class ${\mathcal{C}}^{0,1}$. Then 
\begin{equation}\label{fundamental0}
\operatorname{curl}  u \cdot \nu =0 \ \ {\rm on}\ \partial \Omega,
\end{equation}
for all $u\in  H_0(\cu, \Omega)$ such that $  \operatorname{curl}  u \in H^1(\Omega)$.
\end{lem}

\begin{proof} By integrating by parts and using the well-known formula 
\begin{equation}\label{intparts}
\int_{\Omega}\cu u \cdot F \, dx =\int_{\Omega} u \cdot \cu  F \,dx+\int_{\partial \Omega }(\nu \times  u )\cdot F \, d\sigma\, ,
\end{equation}
we get
\begin{eqnarray*}\lefteqn{
\int_{\partial \Omega} (\operatorname{curl}  u \cdot \nu) \varphi \, d\sigma =\int_{\partial \Omega } 
 (\operatorname{curl}  u \cdot \nu) \varphi \, d\sigma - \int_{ \Omega } 
 \di \operatorname{curl}  u \, \varphi \, dx  }    \\
 &  \qquad = \int_{ \Omega } \operatorname{curl}  u \cdot \nabla \varphi \, dx  
 = \int_{ \Omega } u \cdot \operatorname{curl}    \nabla \varphi \,  dx   
 + \int_{ \partial \Omega }  (\nu \times    u) \cdot \nabla \varphi   \,  d\sigma    =0 ,
\end{eqnarray*}
for all $\varphi \in H^2(\Omega)$, hence by a standard  approximation argument, we deduce that 
$$
\int_{\partial \Omega} (\operatorname{curl}  u \cdot \nu) \varphi \, d\sigma=0
$$
for all $\varphi \in H^1(\Omega)$, which allows to prove the validity of \eqref{fundamental0} by the Fundamental Lemma of the Calculus of Variations. 
\end{proof}

Recall that the electric problem   under consideration is 
\begin{equation}\label{main}
\left\{
\begin{array}{ll}
\cu\, \cu u = \lambda u,& \ \ {\rm in}\ \Omega ,\\
\di u =0, & \ \ {\rm in}\ \Omega ,\\
\nu \times u = 0 ,& \ \ {\rm on}\ \partial \Omega ,
\end{array}
\right.
\end{equation}
which is nothing but problem \eqref{claE} with the precise indication that the vector field $u$ is divergent free. 

It is not difficult to see that the weak formulation of   \eqref{main}  can be written as 
\begin{equation}\label{mainweak}
\int_{\Omega}\cu u \cdot \cu \varphi \, dx =\lambda \int_{\Omega} u \cdot \varphi \, dx, \ \ {\rm for\ all}\ \varphi \in \xn (\di\, 0, \Omega ),
\end{equation}
in the unknowns $u\in \xn (\di\, 0, \Omega )$ and $\lambda \in \R$. 

Since for our purposes we prefer to work in the space $\xn (\Omega)$ rather than in the space $\xn(\di\,  0, \Omega) $, following \cite{coda, cocoercive},  we introduce  a penalty term
in the equation and  we replace problem  \eqref{main} by the problem
\begin{equation}\label{mainpen}
\left\{
\begin{array}{ll}
\cu\, \cu u  -\tau \nabla \di u = \lambda u,& \ \ {\rm in}\ \Omega ,\\
\nu \times u = 0 ,& \ \ {\rm on}\ \partial \Omega ,
\end{array}
\right.
\end{equation}
where $\tau $ is any fixed positive real number.  Problem \eqref{mainpen} is understood in the weak sense as follows:
\begin{equation}
\label{mainpenweak}
\int_{\Omega}\cu u \cdot \cu \varphi \, dx+\tau \int_{\Omega} \di u\,  \di \varphi \,dx=\lambda \int_{\Omega} u \cdot \varphi \, dx, \ \ {\rm for\ all}\ \varphi \in \xn (\Omega ),
\end{equation}
in the unknowns $u\in \xn (\Omega)$ and $\lambda \in \R$.

It is obvious that the solutions of problem \eqref{mainweak} are    the divergence free  solutions of \eqref{mainpenweak}. On the other hand, it is also not difficult to see that the solutions of problem \eqref{mainpenweak} which are not divergence free are given by the vector fields  $u=\nabla f $  of the gradients of the solutions $f$ to the Helmohltz 
equation with Dirichlet boundary conditions, that is 
\begin{equation}\label{hel}
\left\{
\begin{array}{ll}
-\Delta f = \frac{\lambda}{\tau} f,& \ \ {\rm in}\ \Omega ,\\
f = 0 ,& \ \ {\rm on}\ \partial \Omega .
\end{array}
\right.
\end{equation}

In fact, we have the following result from \cite{coda}

\begin{lem}\label{union} Let $\Omega$ be a bounded domain in $\R^3$ of class ${\mathcal{C}}^{0,1}$.  A vector field $u\in \xn (\Omega)$ is a solution of problem \eqref{mainpen} with $\di u=0$ if and only if $u\in \xn (\di\, 0,\Omega )$ is a solution of problem 
\eqref{mainweak}. Moreover, a vector field $u\in \xn (\Omega)$  with $\di u\ne 0$ is a solution of problem \eqref{mainpen}  if and only if $u=\nabla  f$ where $f\in H^1_0(\Omega)$ is a solution of problem \eqref{hel}.  In particular, the set of eigenvalues of problem \eqref{mainpen} are given by the union of the 
set of eigenvalues of problem \eqref{mainweak} and the set of eigenvalues of the Dirichlet Laplacian in $\Omega$ multiplied by $\tau $. 
\end{lem}

In view of the previous lemma,  in order to distinguish the solutions  arising from  the original Maxwell system from the spurious solutions associated with the Helmohltz 
equation,  we give the following definition. 

\begin{defn}\label{risonanza} We say that a couple $(u,\lambda )$ in  $ \xn(\Omega) \times \mathbb{R}$ is a (electric) Maxwell eigenpair if  $(u, \lambda)$ is an eigenpair 
of equation \eqref{mainpenweak} with $\di\, u =0 $ in $\Omega$, in which case $u$ is called a  (electric) Maxwell eigenvector and $\lambda $ a Maxwell eigenvalue. 
\end{defn}

\begin{rem}\label{risonanzarem}     In this paper, it will be understood that the value of $\tau $ in \eqref{mainpen}  is fixed. It is important to note that in applying our results one is free to choose $\tau >0$ in order to avoid 
the overlapping of Maxwell and Helmholtz eigenvalues.  In fact, 
  since the set of eigenvalues of problem \eqref{mainpen} are given by the union of the 
set of eigenvalues of problem \eqref{mainweak} and the set of eigenvalues of the Dirichlet Laplacian in $\Omega$ multiplied by $\tau $, one cannot exclude that
a Maxwell eigenvalue could coincide with an eigenvalue of the Dirichlet Laplacian multiplied by some $\tau \in ]0,\infty [$. However, if $\lambda $ is a fixed 
Maxwell eigenvalue it is possible to choose $\tau \in ]0, \infty[$ such that $\lambda \ne    \tau  \vartheta$  for all eigenvalues $\vartheta$ of the Dirichlet Laplacian, in other words one can choose $\tau $ in order to avoid `resonance'.
It is also useful to recall that the eigenvalues  of the Dirichlet Laplacian depend with continuity upon  sufficiently regular perturbations of   $\Omega$, as those considered in this paper (see e.g., \cite{lala}), hence  it is possible to avoid `resonance'  around a fixed Maxwell eigenvalue $\lambda (\Omega)$, possibly multiple,  and all those eigenvalues bifurcating from it when $\Omega$ is slightly perturbed. 
\end{rem}

We now describe a standard procedure that allows us to recast the eigenvalue problem \eqref{mainpenweak} as an eigenvalue problem for a compact self-adjoint 
operator in Hilbert space. We consider the operator $T$ from $\xn (\Omega )$ to its dual $(\xn (\Omega))'$ defined by the pairing
\begin{equation}\label{operatort}
<Tu, \varphi>= \int_{\Omega}\cu u \cdot \cu \varphi \, dx+\tau \int_{\Omega} \di u\,  \di \varphi \, dx  ,
\end{equation}
for all $u,\varphi \in \xn (\Omega )$. Then, we consider the map 
$J$ from $(L^2(\Omega ))^3$ to $(\xn (\Omega))'$ defined by the pairing
$$
<Ju, \varphi>=\int_{\Omega} u \cdot \varphi \, dx ,
$$
for all $u\in (L^2(\Omega ))^3$ and $\varphi  \in \xn (\Omega )$. 
  By the Riesz Theorem, the operator $T+J$ is a homeomorphism from $\xn (\Omega )$ to its dual. 

\begin{lem}  Let $\Omega$ be a bounded open set in $\R^3$. The operator $S$ from $ (L^2(\Omega ))^3$ to itself defined by
$$
Su=\iota \circ (T+J)^{-1}\circ J
$$
where $\iota$ denotes the embedding of $\xn (\Omega )$  into  $  (L^2(\Omega ))^3$ is a non-negative self-adjoint operator in  $ (L^2(\Omega ))^3$.  Moreover, 
 $\lambda $ is an eigenvalue of problem \eqref{mainpenweak} if and only if $\mu =(\lambda +1)^{-1}$ is an eigenvalue of 
the operator $S$, the eigenfunctions being the same.
\end{lem}

 If the space $\xn (\Omega)$ is compactly embedded into  $ (L^2(\Omega ))^3$, that is, $\iota$ is a compact map, then the operator $S$ is compact, hence the spectrum $\sigma (S)$ of $S$ can be represented as $\sigma (S)=\{0\} \cup \{\mu_n (\Omega ) \}_{n\in \N} $ where 
$\mu_n(\Omega )$, $n\in \N$  is a decreasing sequence of positive eigenvalues of finite multiplicity,  which converges to zero.  
Accordingly, the eigenvalues of problem  \eqref{mainpenweak}  are given by the sequence $\lambda_n(\Omega ) $, $n\in \N$ defined by $
\lambda_n(\Omega )= \mu_n^{-1}(\Omega )-1
$.  As customary, we agree to repeat each eigenvalue in the sequence as many times as its multiplicity.   Thus, we have the following result where formula \eqref{minmax}
can be proved by applying the classical Min-Max Principle to the operator $S$. 

\begin{thm} Let $\Omega$ be a bounded open set such that the embedding $\xn (\Omega )\subset (L^2(\Omega))^3$ is compact.  The eigenvalues of problem 
 \eqref{mainpenweak} have finite multiplicity and  are given by a divergent sequence  $\lambda_n(\Omega )$, $n\in \N$ which can be represented by means of the following min-max formula:
 \begin{equation}\label{minmax}
\lambda_n(\Omega )=   \min_{ \substack{ V\subset \xn(\Omega )  \\ {\rm dim }V=n }  }\  \,  \max _{u\in V\setminus\{0\}  }
  \frac{  \int_{\Omega}  |\cu u|^2 + \tau   |\di u |^2  dx }{\int_{\Omega}   |u|^2\, dx}.
\end{equation}
\end{thm}

\section{Domain transplantation}
\label{transec}

Given a bounded domain  (i.e., a bounded connected open set)  $\Omega$ in $\R^3$, we consider  problem \eqref{mainpenweak} on a class of domains  $\f (\Omega)$ obtained as diffeomorphic
images of $\Omega$. 
 Namely, we consider the family of diffeomorphisms
 $$\mathcal{A}_\Omega = \left\{ {\Phi} \in \mathcal{C} ^{1,1}(\overline{\Omega}, \mathbb{R}^3) : {\Phi} \textnormal{ is injective, } \operatorname{det}\operatorname{D}{\Phi}(x) \neq 0 \ \forall x \in \overline{\Omega} \right\},$$ 
 where  $\mathcal{C}^{1,1}(\overline{\Omega}, \mathbb{R}^3)$ is the space of ${\mathcal{C}}^{1,1}$ functions from $\overline{ \Omega} $ to ${\mathbb{R}}^3$ endowed with 
 its standard norm defined  by  $\| \f \|_{{\mathcal{C}}^{1,1}} = \| \f \|_{\infty}+\| \nabla \f \|_{\infty} + | \nabla\f |_{0,1}$ for all $\f\in \mathcal{C}^{1,1}(\overline{\Omega}, \mathbb{R}^3)$, where 
 $| \cdot |_{0,1}$ denotes the Lipschitz seminorm.  We note that if $\f \in \mathcal{A}_\Omega $ then 
 $\f (\Omega )$ is a bounded domain in $\R^3$,  $\partial {\Phi}(\overline{\Omega}) = {\Phi}(\partial \Omega) = \partial {\Phi}(\Omega),$ and ${\Phi}(\Omega)$ is the interior of ${\Phi}(\overline{\Omega})$. The map ${\Phi}$ is a homeomorphism of $\overline{\Omega}$ onto $\overline{{\Phi}(\Omega)}$. Moreover, if $\Omega$ is of class ${\mathcal{C}}^{1,1}$ then $\Phi (\Omega)$ is also of class ${\mathcal{C}}^{1,1}$. Finally, we recall that if $\Omega$ is sufficiently regular, say of class $C^1$, then  
 $\mathcal{A}_\Omega$ is an open set in $\mathcal{C}^{1,1}(\overline{\Omega}, \mathbb{R}^3)$.  See \cite{lala2002} and \cite{lanza} for details. 

In order to study problem  \eqref{mainpenweak} on  $\f (\Omega)$, it is convenient to pull it back to  $\Omega$ by means of a change of variables. As is known,  in order   to transform the curl in  a natural way and   preserve  our boundary conditions, it is necessary to pull back  any vector field  $v$ defined on $\f (\Omega)$ to the vector field $u$ defined on $\Omega$ by means of the covariant Piola transform defined by 
\begin{equation}\label{pullback}
u(x)= \left((v \circ \f) \operatorname{D}\f \,\right)(x),\ \ {\rm for\ all }\  x\in \Omega . 
\end{equation}
By setting
$$
y=\f (x), \ \ {\rm for\ all }\ x\in \Omega ,
$$
equality \eqref{pullback} can be rewritten in the form
\begin{equation} v(y)= \left( u (\operatorname{D}\f)^{-1} \right) \circ \f^{(-1)}(y) = \left( u \circ \f^{(-1)} \right) \operatorname{D}(\f^{(-1)}) (y),    \ \ y\in \f (\Omega).
\end{equation}

Note that in the sequel we shall often use the following notation
 $$\partial_j u_i(x) = \frac{\partial u_i}{\partial x_j}(x) \quad \text{ and } \quad \partial'_a v_b(y) = \frac{\partial v_b}{\partial y_a}(y).$$

Then we have the following  known result, which can be found for example in  \cite[Corollary~3.58]{monk}. For the convenience of the reader, we include a proof (which  differs from that of \cite[Corollary~3.58]{monk}).  Note that the assumption $\Phi \in {\mathcal{C}}^{1,1}$ can be relaxed, but some care is required, see Remark~\ref{remc1}.

\begin{thm}[Change of variables for curl]
\label{changecurl}
Let $\Omega$ be a bounded domain in $\R^3$ and $\f\in \mathcal{A}_\Omega$. Then
a function $v$ belongs to $H(\cu, \f (\Omega ))$  ($H_0(\cu, \f (\Omega ))$, respectively),  if and only if the function $u$ defined by \eqref{pullback} belongs to $H(\cu, \Omega )$ ($H_0(\cu, \Omega )$, respectively), in which case 
\begin{equation}
(\operatorname{curl}_y v (y)) \circ \f = \frac{\operatorname{curl}_x u (x) \left(\operatorname{D} \f(x) \right)^{T}} {\operatorname{det} \left( \operatorname{D} \f(x) \right)} \, .
\label{changecurl0}
\end{equation}
\end{thm}

{\bf Proof.}   Assume for the time being that 
$ u $ is a vector field of class ${\mathcal{C}}^{0,1}$.  
The chain rule yields 
\begin{equation*}
\begin{split}
 \frac{\partial v_b}{\partial y_a}(y)&= \frac{\partial [(u_i \circ \f^{(-1)}) \partial'_b \f_i^{(-1)}]}{\partial y_a}(y) \\
    &= \frac{\partial u_i}{\partial x_j}(\f^{(-1)}(y)) \, \frac{\partial \f_j^{(-1)}}{\partial y_a}(y) \, \frac{\partial \f_i^{(-1)}}{\partial y_b}(y) + u_i(\f^{(-1)}(y)) \frac{\partial^2 \f_i^{(-1)}}{\partial y_a \partial y_b}(y).
\end{split}
\end{equation*}
Note that  summation symbols are omitted here and in the sequel.
Recall that the $c$-component of the curl of $v$ is given by
$$\left[ \operatorname{curl}_y v(y) \right]_c = \partial'_a v_b (y) \, \xi_{abc} \, ,$$
where $\xi_{abc}$ is the Levi-Civita symbol defined by
\[
    \xi_{abc}=\left\{
                \begin{array}{ll}
                  +1 & \qquad \text{if } (a,b,c) \text{ is an even permutation of }(1,2,3),\vspace{1mm}\\
                  -1 & \qquad \text{if } (a,b,c) \text{ is an odd permutation of }(1,2,3),\vspace{1mm}\\
                   \ \ 0 & \qquad \text{if } a=b, \text{ or } b=c, \text{ or } a=c.
                 \end{array}
              \right.
\]
Then 
\begin{equation*}
\begin{split}
\left[ \operatorname{curl}_y v(y) \right]_c &= \partial_j u_i(\f^{(-1)}(y)) \, \partial'_a \f_j^{(-1)}(y) \, \partial'_b \f_i^{(-1)}(y) \, \xi_{abc}\\
& \qquad + u_i(\f^{(-1)}(y)) \partial'_a \partial'_b \f_i^{(-1)}(y) \, \xi_{abc}.
\end{split}
\end{equation*}
Since $\xi_{abc}=-\xi_{bac}  $ we have that   for all $i=1,2,3$
$$\sum_{a,b=1}^3 \partial'_a \partial'_b \f_i^{(-1)} (y) \, \xi_{abc} = 0.$$
Thus
\begin{equation*}
\begin{split}
\left[ \operatorname{curl}_y v (y) \left( \operatorname{D} \f^{(-1)}(y)\right)^T \right]_k &= \left[\operatorname{curl}_y v(y) \right]_c \frac{\partial \f_k^{(-1)}}{\partial y_c}(y) \\
&= \partial_j u_i(\f^{(-1)}(y)) \, \partial'_a \f_j^{(-1)}(y) \, \partial'_b \f_i^{(-1)}(y) \, \partial'_c \f_k^{(-1)}(y) \, \xi_{abc}\\
&= \partial_j u_i (\f^{(-1)}(y)) \, \xi_{jik} \operatorname{det}\left( \operatorname{D}\f^{(-1)}(y)\right) \\
&= \left( \operatorname{curl}_x u (\f^{(-1)}) \right)_k \operatorname{det}\left( \operatorname{D}\f^{(-1)}(y)\right),
\end{split}
\end{equation*}
where we have used the fact that  $\partial'_a F_j \, \partial'_b F_i \, \partial'_c F_k \, \xi_{abc}= \xi_{jik} \operatorname{det}\left( \operatorname{D}F \right)$, for any vector field $F$ of class $\mathcal{C}^1$.  

Therefore $$\operatorname{curl}_y v (y) = \operatorname{curl}_x u (\f^{(-1)}(y)) \left(\operatorname{D} \f^{(-1)}(y) \right)^{-T} \operatorname{det}\left( \operatorname{D}\f^{(-1)}(y)\right),$$
and formula \eqref{changecurl0} follows.

We now prove the validity of formula \eqref{changecurl0}  in the weak sense. 
We begin with proving that  if  $v\in H(\cu, \f (\Omega ))$ then the  distributional curl of the function  $u$ defined above  belongs to $L^2(\Omega )$ and satisfies formula
\eqref{changecurl0}. To do so, it suffices to prove that for 
\begin{equation}\label{changecurl1}
\int_{\Omega}  u\, (\cu \varphi )^T dx =\int_{\Omega} (\operatorname{curl}_y v (y)) ( \f (x))    \left(\operatorname{D} \f(x) \right)^{-T} \operatorname{det}  \operatorname{D} \f(x) 
\varphi^T (x) \, dx  ,
\end{equation}
for all $\varphi \in {\mathcal{C}}^{\infty }_c(\Omega)$.
Following formula \eqref{pullback}, we define a function $\psi $ on $\Phi (\Omega)$ by setting  
\begin{equation}\label{changecurl15}
\varphi (x)= \left((\psi  \circ \f) \operatorname{D}\f \,\right)(x)\, .
\end{equation}
By formula  \eqref{changecurl0} we get
$$
\cu_x \varphi (x) = 
(\operatorname{curl}_y \psi (y)) ( \f (x))    \left(\operatorname{D} \f(x) \right)^{-T} \operatorname{det}  \operatorname{D} \f(x)
$$
and this implies that 
\begin{eqnarray*}\lefteqn{
\int_{\Omega}  u\, (\cu \varphi )^Tdx   =\int_{\Omega} u  \left(\operatorname{D} \f(x) \right)^{-1}  (\operatorname{curl}_y \psi (y))^T ( \f (x))   \operatorname{det}  \operatorname{D} \f(x) dx}\nonumber  \\
& & =\int_{\Omega } v(\Phi (x))(\operatorname{curl}_y \psi (y))^T ( \f (x))   \operatorname{det}  \operatorname{D} \f(x) dx \\
& & =
\int_{\f (\Omega )} v(y)(\operatorname{curl}_y \psi (y))^T {\rm sgn}( \operatorname{det}  \operatorname{D}\Phi^{(-1)} (y)   )  dy    \nonumber \\
& & =\int_{\f (\Omega )} (\operatorname{curl}_y v(y)) \psi^T   {\rm sgn}( \operatorname{det}  \operatorname{D}\Phi^{(-1)} (y)   ) dy\\
& & = \int_{\Omega} (\operatorname{curl}_y v (y)) ( \f (x))    \left(\operatorname{D} \f(x) \right)^{-T} \operatorname{det}  \operatorname{D} \f(x) 
\varphi^T (x) dx,
\end{eqnarray*}
as required.  In the same way, one can prove that   if  $u\in H(\cu, \Omega )$ then the  distributional curl of the function  $v$  belongs to $L^2(\f(\Omega ))$, which completes the first part of the proof. 

In order to prove that $v$ belongs to $H_0(\cu, \f (\Omega ))$  if and only $u$  belongs to $H_0(\cu, \Omega )$ one can  directly use formula \eqref{changecurl0} and the  definition of the these spaces. 
\hfill $\Box$\\

\begin{rem}Theorem~\ref{changecurl} holds also under weaker assumptions on $\Phi$. Namely, assume that  $\f\in  \mathcal{C}^{1}(\overline{\Omega}, \mathbb{R}^3) $ is injective and that  $\operatorname{det}\operatorname{D}{\Phi}(x) \neq 0 \ \forall x \in \overline{\Omega} $.  Then the thesis of Theorem~~\ref{changecurl} holds. 
Indeed, given any smooth domain $U$ with $\overline{U}\subset  \Omega$, 
one can find  an approximating sequence $\Phi_n\in {\mathcal{A}}_{U}$, $n\in \mathbb{N}$ which converges 
to $\Phi$ in $\mathcal{C}^{1}(\overline{U}, \mathbb{R}^3) $. This can be done by using standard 
mollifiers. Then, since  the set of functions $ \mathcal{C}^{1}(\overline{U}, \mathbb{R}^3)  $ which are injective and such that $\operatorname{det}\operatorname{D}{\Phi}(x) \neq 0 \ \forall x \in \overline{U} $, is an open set in  $\mathcal{C}^{1}(\overline{U}, \mathbb{R}^3) $ (cfr., \cite[Lemma~5.2]{lanza}), it follows that $\f_n\in {\mathcal{A}}_{U} $ for all $n$ sufficiently large, hence Theorem~\ref{changecurl} is applicable to $\f_n$.  Passing to the  limit as $n\to \infty $ we get the validity of formula \eqref{changecurl0} in $U$, and since $U$ is arbitrary, formula  \eqref{changecurl0} holds also in the whole of $\Omega$. The preservation of the spaces easily follows by formula \eqref{changecurl0} itself and changing variables in integrarls.
\end{rem}

\begin{rem}\label{remc1} The fact that  $v$ belongs to $H_0(\cu, \f (\Omega ))$  if and only if the function $u$ defined by \eqref{pullback} belongs to $H_0(\cu, \Omega )$ as stated in 
Theorem~\ref{changecurl} has a immediate explanation by using traces in the classical sense as follows. 
It is not difficult to realise that the unit outer normals 
 to $\partial\Omega$ and $\partial\f(\Omega)$ satisfy the relation 
$$\nu_{ \partial\f(\Omega)} \circ \f =\pm\,  \frac{\nu_{\partial \Omega} (\operatorname{D}\f)^{-1}}{\abs{\nu_{ \partial \Omega} (\operatorname{D}\f)^{-1}}}.$$
\noindent
Then, using the fact that $aM \times bM = \operatorname{det(M) (a \times b)} (M)^{-1} $ for all vectors $a,b \in \mathbb{R}^3$ and for all invertible matrices $M \in GL_3(\mathbb{R})$, we immediately deduce that 
\begin{equation}\label{preservation }   v \times \nu_{\scriptscriptstyle \partial \f(\Omega)}=0 \text{ on } \partial \f(\Omega)\ \ {\rm if\ and\ only\ if}\ \  u \times \nu_{\scriptscriptstyle \partial \Omega}=0 \text{ on } \partial \Omega ,
\end{equation}
for all vector fields admitting boundary values in the classical sense. 
\end{rem}

In order to transplant  problem \eqref{mainpenweak} from $\f (\Omega)$ to $\Omega$ we also need a formula for the  transformation of the  divergence under the action of the  pull-back operator defined in \eqref{pullback}. 

\begin{thm}
\label{changediv}
Let $\Omega$ be a bounded domain in $\R^3$ and $\f\in \mathcal{A}_\Omega$. Then  
a function $v$ belongs to $H^1( \f (\Omega ))$   if and only if the function $u$ defined by \eqref{pullback} belongs to $H^1(\Omega )$, in which case 
 \begin{equation}
 \label{changediv0}
 (\displaystyle \operatorname{div}_y v) \circ \f (x)  = \frac{\operatorname{div}_x \left[ u (x) (\operatorname{D}\f(x))^{-1} (\operatorname{D}\f(x))^{-T} \operatorname{det}(\operatorname{D} \f(x)) \right]}{\operatorname{det} (\operatorname{D} \f(x))}.
 \end{equation}
\end{thm}

{\bf Proof.} The first part of the statement is standard and can be carried out by using the chain rule and changing variables in integrals. The proof of formula \eqref{changediv0} is more involved. To simplify notation, we set  $M= \operatorname{D} \f^{(-1)}$, so that $\partial'_a = M_{i,a} \partial_i$, where $\displaystyle M_{i,a}=\partial \f_i^{(-1)} /  \partial  y_a$.  Note that $M_{j,a} = \sum_{m,k=1}^3 M_{j,m} M_{k,m} (M^{-1})_{a,k}$ simply because 
$M_{j,a} = ( M M^T M^{-T} )_{j,a}$.
Since $v_a = (u_j \circ \f^{(-1)} ) \, \partial'_a \f_j^{(-1)}= (u_j \circ \f^{(-1)} ) M_{j,a}$ we have that
\begin{equation*}
\begin{split}
\operatorname{div}_y v &= \partial'_a v_a =   \partial'_a [ \left(u_j \circ \f^{(-1)}\right) M_{j,a} ] \\
&=\partial'_a [ ( u_j \circ \f^{(-1)} ) M_{j,m} M_{k,m} (M^{-1})_{a,k} ] \\
&=\partial'_a [ ( u_j \circ \f^{(-1)} ) (\partial'_m \f_j^{(-1)})(\partial'_m \f_k^{(-1)}) ( (\partial_k \f_a) \circ \f^{(-1)} ) ]\\
&= \partial'_a [P \,  Q ]= (\partial'_a P) Q + P (\partial'_a Q)
\end{split}
\end{equation*}
 where  we have set $P=P(k) = \sum_{j,m=1}^3 \displaystyle   ( u_j \circ \f^{(-1)} ) (\partial'_m \f_j^{(-1)})(\partial'_m \f_k^{(-1)})    \operatorname{det}^{-1}( \operatorname{D} \f^{(-1)} )       $ and 
$Q=Q(k,a) $ $ = ( (\partial_k \f_a) \circ \f^{(-1)} ) \operatorname{det}( \operatorname{D} \f^{(-1)} ).$
We claim that $\sum_{a=1}^3 \partial'_a Q = 0$. Indeed, if by $C$ we denote the cofactor matrix of $M$, we have that (see \cite{mccon}, p.12)
$$C_{k,a} = \frac{1}{2}\sum_{n,m,i,j=1}^3 \xi_{anm} \, \xi_{kij} M_{i,n} M_{j,m} \,,$$
hence 
\begin{equation*}
\begin{split}
\partial'_a \left( (M^{-1})_{a,k} \operatorname{det}(\operatorname{D}\f^{(-1)}   ) \right) &= \partial'_a (C_{k,a}) = \frac{1}{2} \partial'_a (\xi_{anm} \, \xi_{kij} M_{i,n} M_{j,m}) \\
&= \frac{1}{2} \xi_{kij} M_{j,m} \left( \partial'_a M_{i,n} \right) \xi_{anm} + \xi_{anm} M_{i,n} \left( \partial'_a M_{j,m} \right) \xi_{kij}.
\end{split}
\end{equation*}
Moreover
\begin{equation*}
\begin{split}
\sum_{a,n=1}^3 \xi_{anm} \, \partial'_a M_{i,n} &= \sum_{a,n=1}^3 \xi_{anm} \, \partial'_n M_{i,a} = \sum_{a,n=1}^3 \xi_{nam} \, \partial'_a M_{i,n} 
= - \sum_{a,n=1}^3 \xi_{anm} \, \partial'_a M_{i,n}.
\end{split}
\end{equation*}
Thus $\sum_{a,n=1}^3 \xi_{anm} \, \partial'_a M_{i,n} = 0$.
Similarly $\sum_{a,m=1}^3 \xi_{anm} \, \partial'_a M_{j,m} =0 $ and the claim is proved. 
Then 
\begin{equation*}
\begin{split}
\operatorname{div}_y v & = (\partial'_a P)Q = \operatorname{det} \left( \operatorname{D} \f^{(-1)} \right) (M)^{-1}_{a,k} \partial'_a P  \\ 
& = \operatorname{det} \left( \operatorname{D} \f^{(-1)} \right) (M)^{-1}_{a,k} M_{i,a}  [   \partial_i( P\circ\Phi )]\circ \Phi^{(-1)}  \\
&  = \operatorname{det} \left( \operatorname{D} \f^{(-1)} \right) \delta_{i,k}   [   \partial_i( P\circ\Phi )]\circ \Phi^{(-1)}  \\
& = \operatorname{det} \left( \operatorname{D} \f^{(-1)} \right) \partial_i \left[ \frac{\left( u_j \circ \f^{(-1)} \right) (\partial'_m \f_j^{(-1)})(\partial'_m \f_i^{(-1)})}{\operatorname{det}\left( \operatorname{D} \f^{(-1)} \right)} \circ \Phi    \right]\circ\Phi^{(-1)}  \\
&= \left[    \frac{\partial_i \left[ u_j \left( (\operatorname{D}\f)^{-1} (\operatorname{D}\f)^{-T} \right)_{j,i} \operatorname{det}(\operatorname{D} \f) \right]}{\operatorname{det} (\operatorname{D} \f)}              \right] \circ \f^{(-1)} \\
&= \left[ \frac{\partial_i \left[ \left( u (\operatorname{D}\f)^{-1} (\operatorname{D}\f)^{-T} \right)_i \operatorname{det}(\operatorname{D} \f) \right]}{\operatorname{det} (\operatorname{D} \f)} \right] \circ \f^{(-1)} \\
&= \frac{\operatorname{div}_x \left[ u (\operatorname{D}\f)^{-1} (\operatorname{D}\f)^{-T} \operatorname{det}(\operatorname{D} \f) \right]}{\operatorname{det} (\operatorname{D} \f)} \circ \f^{(-1)},\\
\end{split}
\end{equation*}
hence formula \eqref{changediv0} is proved. \hfill $\Box$\\

We consider equation \eqref{mainpenweak} on $\f (\Omega)$ that is 

\begin{equation}
\label{mainpenweaphi}
\int_{\f (\Omega )}\cu v \cdot \cu \psi \, dy+\tau \int_{\f (\Omega )} \di v\,  \di \psi \,dx=\lambda \int_{\f (\Omega )} v \cdot \psi \, dx, \ \ {\rm for\ all}\ \psi \in \xn ( \f (\Omega )),
\end{equation}
in the unknowns $v\in \xn (\f (\Omega))$ and $\lambda \in \R$.  
If $u$ is the function defined in $\Omega$ by formula \eqref{pullback} and, analogously, $\varphi $ is the function defined by \eqref{changecurl15}, by changing variables in \eqref{mainpenweaphi} we get

\begin{align}\label{mainback}
    & \int_\Omega   \frac{  \operatorname{curl}u \,    (\operatorname{D}\f)^T \operatorname{D}\f      (\operatorname{curl}\varphi)^T } {\abs{\operatorname{det}(\operatorname{D}\f)}}     dx\nonumber \\
    & \quad + \tau \int_\Omega \frac{\operatorname{div}_x \left( u (\operatorname{D}\f)^{-1} (\operatorname{D}\f)^{-T} \operatorname{det}(\operatorname{D}\f)  \right) \operatorname{div}_x \left(\varphi (\operatorname{D}\f)^{-1} (\operatorname{D}\f)^{-T} \operatorname{det}(\operatorname{D}\f)  \right)}{\abs{\operatorname{det}(\operatorname{D}\f)}} \, dx\nonumber \\
    &\qquad\qquad\qquad\qquad \qquad\qquad\qquad\qquad\qquad\qquad=\lambda \int_\Omega u (\operatorname{D}\f)^{-1} (\operatorname{D} \f)^{-T} \varphi^T \abs{\operatorname{det}\operatorname{D}\f} dx\, .
\end{align}

Thus, instead of studying problem \eqref{mainpenweaphi} in the varying domain $\f (\Omega)$, we can study problem \eqref{mainback}  where  the unknown $u\in \xn (\Omega )$  is defined on the fixed domain $\Omega$ and the test functions $\varphi$ have to be taken in the fixed space  $\xn (\Omega )$ as well. Recall that under our regularity assumptions on $\Omega$, the space $\xn (\Omega )$ is contained in $(H^1(\Omega ))^3$. 

It is clear that the natural $L^2$-space for problem \eqref{mainback}   is the usual $L^2$-space endowed with
the scalar product  $ \langle \cdot, \cdot \rangle_\f $ defined by 
\begin{equation}\label{scalar}
 \langle u , \varphi \rangle_\f = \int_\Omega u (\operatorname{D}\f)^{-1} (\operatorname{D} \f)^{-T} \varphi^T \abs{\operatorname{det}\operatorname{D}\f} dx\, ,
\end{equation}
for all $u,\varphi \in (L^2 (\Omega))^3 $, which is equivalent to the usual one.
We denote by $L^2_\f(\Omega)$ the space $(L^2 (\Omega))^3$ endowed with  scalar product \eqref{scalar}.   As we have done for equation \eqref{mainpenweak} we recast 
 problem \eqref{mainback} as a problem for a compact self-adjoint operator. To do so, we consider the operator $T_{\f}$ from  the space $\xn (\Omega )$ to its dual by 
setting  $\langle T_{\Phi}u, \varphi \rangle $ equal to the left-hand side of equation \eqref{mainback}.  In the same way, we define the operator $J_{\Phi}$ from 
$L^2_{\Phi}(\Omega) $ to the dual of $\xn (\Omega )$ by setting $\langle J_{\Phi}u, \varphi \rangle $ equal to the right-hand side of equation \eqref{mainback} divided by $\lambda$.

\begin{lem}  Let $\Omega$ be a bounded domain in $\R^3$. The operator $S_{\Phi}$ from $ L^2_{\Phi}(\Omega )$ to itself defined by
$$
S_{\Phi}u=\iota \circ (T_{\Phi} +J_{\Phi} )  ^{-1}\circ J_{\Phi}
$$
where $\iota$ denotes the embedding of $\xn (\Omega )$  into  $  L^2_{\Phi}(\Omega )$, is a non-negative self-adjoint operator in  $L^2_{\Phi}(\Omega )$.  Moreover, 
 $\lambda $ is an eigenvalue of problem \eqref{mainback} if and only if $\mu =(\lambda +1)^{-1}$ is an eigenvalue of 
the operator $S_{\Phi}$, the eigenfunctions being the same.
\end{lem}

Clearly, if the space $\xn (\Omega)$ is compactly embedded into  $ (L^2(\Omega ))^3$ then the operator $S_{\Phi}$ is compact and its  spectrum is given by  $\sigma (S_{\Phi})=\{0\} \cup \{\mu_n(\f )\}_{n\in \N} $ where 
$$
 \mu_{n}[  \Phi ]=(\lambda_n[\Phi ] +1)^{-1}
$$
and  $\lambda_n[\Phi ]:=\lambda_n(\Phi (\Omega ))$ are the eigenvalues of problem \eqref{mainpenweaphi}.

\section{Analyticity results and Hadamard-type formulas}
\label{analitic}

Given a finite set of indices $F \subset \mathbb{N}$, we consider
$$\mathcal{A}_\Omega[F] \colon = \set{\f \in \mathcal{A}_\Omega : \lambda_j[\f] \neq \lambda_l[\f],\ \forall j \in F, l \in \mathbb{N}\setminus F}$$
and the  elementary symmetric functions of the corresponding eigenvalues
\begin{equation}
\Lambda_{F,s}[\f] = \ \sum_{\substack{j_1,\dots, j_s \in F\\j_1<\dots<j_s}} \lambda_{j_1}[\f] \cdot \cdot \cdot \lambda_{j_s}[\f], \qquad s=1,\dots,\abs{F}.
\end{equation}
It is also convenient to consider 
\begin{equation}
    \hat{\Lambda}_{F,s}[\f]=\sum_{\substack{j_1,\dots, j_s \in F\\j_1<\dots<j_s}} (\lambda_{j_1}[\f]+1)\cdot \cdot \cdot (\lambda_{j_s}[\f]+1),
\end{equation}
for all $\f \in \mathcal{A}_\Omega[F]$ and to note that 
\begin{equation} \label{lambdatilde}
    \Lambda_{F,s}[\f] = \sum_{k=0}^s (-1)^{s-k} \binom{\abs{F}-k}{s-k} \hat{\Lambda}_{F,k}[\f],
\end{equation}
where we have set $\Lambda_{F,0}= \hat{\Lambda}_{F,0}=1$.

Finally, we set
$$\Theta_\Omega [F]\colon = \set{\f \in \mathcal{A}_\Omega[F] : \lambda_j[\f ]\  \emph{\rm have a common value } \lambda_F[\f] \ \forall j \in F}.$$

Then, we can state our main theorem the proof of which is also based on Lemma~\ref{techlemma} below.  Concerning the assumption on the summability of the second order derivatives of the eigenvectors, we include the following remark.

\begin{rem}\label{regularity}  If $\Omega$ is of class $C^{2,1}$ (which means that locally at the boundary $\Omega$ can be described by the subgraphs of functions of class $C^2$ with Lipschitz continuous second order derivatives), hence in particular if $\Omega$ is of class $C^3$, then   the eigenvectors of problem \eqref{mainweak} belong to the standard Sobolev space $H^2(\Omega)$ of functions in $L^2(\Omega)$ with weak derivatives up to the second order in $L^2(\Omega)$. See e.g., \cite{weber} or the more recent paper \cite{alberti}. 
 \end{rem}

\begin{thm}\label{symmetrican}
    Let $\Omega$ be a bounded domain in  $\mathbb{R}^3$ of class ${\mathcal{C}}^{1,1}$. Let $F$ be a finite non-empty subset of $\mathbb{N} $. 
    Then  $\mathcal{A}_\Omega[F]$ is an open set in $\mathcal{C}^{1,1}(\overline{\Omega}, \mathbb{R}^3)$ and  $\Lambda_{F,s}[\f]$ depends real-analytically on $\f \in \mathcal{A}_\Omega[F]$.  
    
    Let $\tf \in \Theta_\Omega [F]$. Assume that $ \lambda_{F}[\tilde \f] $ is a Maxwell eigenvalue and    $\tilde{E}^{(1)},\dots, \tilde{E}^{(\abs{F})}\in \xn (\di \, 0, \tilde\Phi(\Omega) )$  is  an orthonormal basis of  Maxwell eigenvectors for the corresponding eigenspace, where the orthonormality is taken in $(L^2(\tf (\Omega ))^3$, and assume that those eigenvectors belong to $H^2(\tilde\Phi(\Omega))$.  Then for any $s \in \set{1, \dots, \abs{F}}$,  we have
    \begin{equation}\label{hadamard}
    \begin{split}
        &d\rvert_{\f=\tf}(\Lambda_{F,s})[\Psi]\\
        & \qquad= \binom{\, \abs{F}-1}{s-1} (\lambda_F[\tf])^{s-1} \sum_{l=1}^{\abs{F}} \inttpos \left( \lambda_F [\tf] \abs{\tilde{E}^{(l)}}^2 - \abs{\operatorname{curl} \tilde{E}^{(l)}}^2 \right) \left( (\Psi \circ \tfi) \cdot \nu \right) \, d\sigma
        \end{split}
    \end{equation}
    for all $\Psi  \in  \mathcal{C}^{1,1}(\overline{\Omega}, \mathbb{R}^3) $. 
\end{thm}
\begin{proof} Recall that $\mu_j[\Phi ] =  (\lambda_j[\f]+1)^{-1}$, $j\in {\mathbb{N}}$,  are the eigenvalues of the operator $S_{\Phi }$, hence the set 
$\mathcal{A}_\Omega[F]$ coincides with the set $  \{\f \in \mathcal{A}_\Omega : \mu_j[\f] \neq \mu_l[\f],\ \forall j \in F, l \in \mathbb{N}\setminus F\}$.   Moreover, $S_{\Phi}$ is a compact self-adjoint operator in $L^2_{\Phi}(\Omega )$. Since both the operator $S_{\Phi}$ and the scalar product 
    $\langle \cdot , \cdot \rangle_\f $ depend real-analytically on $\Phi$, being obtained by compositions and inversions  of real-analytic maps (such as linear and multilinear continuous maps), we can apply the general result \cite[Thm.~2.30]{lala}. Thus, the set  $  \set{\f \in \mathcal{A}_\Omega : \mu_j[\f] \neq \mu_l[\f],\ \forall j \in F, l \in \mathbb{N}\setminus F}$ is open in $\mathcal{C}^{1,1}(\overline{\Omega}, \mathbb{R}^3)$ as required, and the functions  
$$M_{F,s}[\Phi ] \colon = \sum_{\substack{j_1,\dots, j_s \in F \\ j_1 < \dots < j_s }} \mu_{j_1}[\f] \cdots \mu_{j_s}[\f],$$
depend real-analytically on $\Phi\in \mathcal{A}_\Omega[F]$. Since 
  $$\hat{\Lambda}_{F,s}[\f] = \frac{M_{F,\, \abs{F}-s}[\Phi   ]   }{M_{F,\, \abs{F}}[\Phi  ]}, \qquad s=1,\dots, \abs{F},$$
  we deduce that $\hat{\Lambda}_{F,s}[\f] $ depend real-analytically on $\Phi\in \mathcal{A}_\Omega[F]$. Finally, by formula \eqref{lambdatilde} we conclude that 
$    \Lambda_{F,s}[\f]$ depend real-analytically on $\Phi\in \mathcal{A}_\Omega[F]$. 

We now prove formula \eqref{hadamard}.     We set $\tu^{(l)} = (\tilde{E}^{(l)} \circ \tf) \operatorname{D}\tf$ for all $l=1,\dots, \, \abs{F}$ and we observe that $\tilde u^{(l)},\ l\in 1, \dots , |F|$ is an orthonormal basis  in  $L^2_{\Tilde \Phi }(\Omega )$ of the eigenspace associated with the eigenvalue $\mu_{F}[\tf]:=(\lambda_F[\tf]+1)^{-1}$ of the operator $S_{\tilde \Phi }$. 
 By applying \cite[Thm.~2.30]{lala},    we have that
    \begin{equation*}
        d\rvert_{\f=\tf} M_{F,s} [\Psi ] = \binom{\, \abs{F}-1}{s-1} (\lambda_F [\tf]+1)^{1-s} \sum_{l=1}^{\abs{F}} \langle d\rvert_{\f=\tf} \, S_{\f} [\Psi][\tu^{(l)}], \tu^{(l)} \rangle_{\tf}.
    \end{equation*}
    Therefore, by Lemma~\ref{techlemma}
    \begin{align*}
        &d\rvert_{\f=\tf}(\hat{\Lambda}_{F,s}) [\Psi] \\
        &  =\{d\rvert_{\f=\tf} M_{F, \, \abs{F}-s} \, [\f   ][\Psi] M_{F, \, \abs{F}} \, [\tf   ]    - M_{F, \, \abs{F}-s} \, [\tf] \, d\rvert_{\f=\tf} M_{F, \, \abs{F}} \, [\f] [\Psi] \} \\
        &\cdot (\lambda_F [\tf]+1)^{2 \, \abs{F}}  = \left[ \binom{\, \abs{F}-1}{\, \abs{F}-s-1} (\lambda_F [\tf] +1)^{s+1-2 \, \abs{F}}- \binom{\, \abs{F}}{s} \binom{\, \abs{F}-1}{\, \abs{F}-1} \right. \\
        & \left. (\lambda_F [\tf] +1)^{s+1-2 \, \abs{F}} \right] (\lambda_F [\tf]+1)^{2 \, \abs{F}} \sum_{l=1}^{\abs{F}} \langle d\rvert_{\f=\tf} \, S_{\f} [\Psi][\tu^{(l)}], \tu^{(l)} \rangle_{\tf}\\
        &  = (\lambda_F [\tf]+1)^{s-1} \binom{\, \abs{F}-1}{s-1} \sum_{l=1}^{\abs{F}} \inttpos \left( \lambda_F [\tf] \abs{\tilde{E}^{(l)}}^2 - \abs{\operatorname{curl} \tilde{E}^{(l)}}^2 \right) \left( (\Psi \circ \tfi) \cdot \nu \right) \, d\sigma.
    \end{align*}
    Thus, using \eqref{lambdatilde}, we get
    \begin{eqnarray*}\lefteqn{
        d\rvert_{\f=\tf}(\Lambda_{F,s}) [\Psi]}\\
        &  &= \sum_{k=1}^s (-1)^{s-k} (\lambda_F [\tf] +1)^{k-1} \binom{\, \abs{F} - k}{s-k} \binom{\, \abs{F}-1}{k-1}\\  
        & & \quad \cdot \sum_{l=1}^{\abs{F}}\inttpos \left( \lambda_F [\tf] \abs{\tilde{E}^{(l)}}^2 - \abs{\operatorname{curl} \tilde{E}^{(l)}}^2 \right) \left( (\Psi \circ \tfi) \cdot \nu \right) \, d\sigma\\
        & & = \binom{\, \abs{F}-1}{s-1} \sum_{k=0}^{s-1} \binom{s-1}{k} (\lambda_F[\tf] +1)^k (-1)^{s-k-1} \\
        & & \quad \cdot  \sum_{l=1}^{\abs{F}} \inttpos \left( \lambda_F [\tf] \abs{\tilde{E}^{(l)}}^2 - \abs{\operatorname{curl} \tilde{E}^{(l)}}^2 \right) \left( (\Psi \circ \tfi) \cdot \nu \right) \, d\sigma \\
        & & = \binom{\, \abs{F}-1}{s-1} (\lambda_F[\tf])^{s-1} \sum_{l=1}^{\abs{F}} \inttpos \left( \lambda_F [\tf] \abs{\tilde{E}^{(l)}}^2 - \abs{\operatorname{curl} \tilde{E}^{(l)}}^2 \right) \left( (\Psi \circ \tfi) \cdot \nu \right) \, d\sigma 
    \end{eqnarray*}
which proves formula \eqref{hadamard}. 
\end{proof}

\begin{lem}\label{techlemma}   Let $\Omega$ be a bounded domain in  $\mathbb{R}^3$ of class ${\mathcal{C}}^{1,1}$ and  $\tilde \Phi \in \mathcal{A}_\Omega$.  Let $\tv, \tw \in \xn (\di \, 0, \tilde \Phi (\Omega))$ be two  Maxwell eigenvectors associated with a Maxwell eigenvalue $\tilde \lambda$. Assume that 
$\tv, \tw \in H^2(\tilde \Phi (\Omega))$.
Let $\tu= (\tv \circ \tilde\f) \operatorname{D}\tilde\f $, $\te = (\tw \circ \tilde\f) \operatorname{D}\tilde\f $.
Then\begin{eqnarray}\lefteqn{
\langle d\rvert_{\f=\tf} S_{\Phi}  [\Psi][\tu], \te \rangle_{\tf} }\nonumber \\
& & \qquad= (\tilde{\lambda}+1)^{-2} \int_{\partial \tf(\Omega)} \left( \operatorname{curl}\tv \cdot \operatorname{curl}\tw - \tilde{\lambda} \ \tv \cdot \tw \right) \left( (\Psi \circ \tfi) \cdot \nu \right) d\sigma \, ,
\end{eqnarray}
for all $\Psi \in  \mathcal{C}^{1,1}(\overline{\Omega}, \mathbb{R}^3) $.
\end{lem}
\begin{proof}
To shorten our notation, we set $\Upsilon_{\Phi}=T_{\Phi}+J_{\Phi}   $.  Since $J_\tf [\tu] = (\tilde{\lambda} +1)^{-1}  \Upsilon_{\tilde\Phi}    [\tu]$, $J_\tf [\te] = (\tilde{\lambda} +1)^{-1} \Upsilon_{\tilde\Phi}  [\te]$, and $ \Upsilon_{\tilde\Phi} $ is symmetric, we have that

\begin{eqnarray}\label{diff1}
\lefteqn{
    \langle  d\rvert_{\f=\tf}   S_{\Phi}  [\Psi][\tu], \te \rangle_{\tf}} \nonumber \\
  &   &   =\langle \iota   \circ\optrit \circ d\rvert_{\f=\tf} J_{\f} [\Psi][\tu], \te \rangle_{\tf} +
    \langle \iota \circ d\rvert_{\f=\tf} \optri [\Psi] \circ J_{\tf} [\tu], \te \rangle_{\tf}\nonumber  \\
  &  &=J_{\tf}[\te] \left[ \iota  \circ\optrit \circ d\rvert_{\f=\tf} J_{\f}[\Psi][\tu] \right] +
    J_{\tf}[\te] \left[ \iota \circ d\rvert_{\f=\tf} \optri [\Psi] \circ J_{\tf} [\tu] \right] \nonumber \\
   & &= (\tilde{\lambda}+1)^{-1} \optrt[\te] \left[\optrit \circ d\rvert_{\f=\tf} J_{\f}[\Psi][\tu] -\optrit \circ d\rvert_{\f=\tf} \optr[\Psi] \circ\optrit \circ  J_{\tf} [\tu] \right]\nonumber  \\
  &  &=(\tilde{\lambda}+1)^{-1} \optrt \left[\optrit \circ d\rvert_{\f=\tf} J_{\f}[\Psi][\tu] -\optrit \circ d\rvert_{\f=\tf} \optr[\Psi] \circ\optrit \circ (\tilde{\lambda}+1)^{-1} \optrt[\tu]\right] [\te] \nonumber \\
 &   &=(\tilde{\lambda}+1)^{-1} \left( d\rvert_{\f=\tf} J_{\f}[\Psi][\tu] [\te] - (\tilde{\lambda}+1)^{-1}d\rvert_{\f=\tf} \optr[\Psi] [\tu] [\te] \right)\nonumber  \\
 &   &= (\tilde{\lambda}+1)^{-1} \left( \tilde{\lambda}(\tilde{\lambda}+1)^{-1} \, d\rvert_{\f=\tf} J_{\f}[\Psi][\tu] [\te] - (\tilde{\lambda}+1)^{-1} d\rvert_{\f=\tf} T_{\f}[\Psi] [\tu] [\te] \right) \nonumber \\
 &   &=(\tilde{\lambda}+1)^{-2} \left( \tilde{\lambda}\, d\rvert_{\f=\tf} J_{\f}[\Psi][\tu] [\te] - d\rvert_{\f=\tf} T_{\f}[\Psi] [\tu] [\te] \right).
\end{eqnarray}

We claim that 
\begin{equation}\label{derjphilemma}
\begin{split}
 d\rvert_{\f=\tf} J_{\f}[\Psi][\tilde u] [\te] = & -\int_{\tf (\Omega)} \tv \left( \operatorname{D}(\Psi \circ \tfi) + \operatorname{D}(\Psi \circ \tfi)^T \right) \tw^T dy \\
 & \qquad \qquad +\int_{\tf(\Omega)} \tv \tw^T \operatorname{div} (\Psi \circ \tfi) dy.
\end{split}
\end{equation}

Indeed, it is easy to see that
\begin{equation}\label{derdet}
    \left[\left[d\rvert_{\f=\tf} (\operatorname{det} (\operatorname{D}\f)) [\Psi] \right] \circ \tfi \right] \operatorname{det}\operatorname{D}\tfi = \operatorname{div}\left( \Psi \circ \tfi \right).
\end{equation}
Moreover, since 
\begin{equation*}
    J_\f[\tu][\te] =\int_\Omega \tu \, R_\f \, \te^T \abs{\operatorname{det}\operatorname{D}\f} dx,
\end{equation*}
where $R_\f = (\operatorname{D}\f)^{-1} (\operatorname{D} \f)^{-T}$,
then
\begin{equation}\label{derjphi}
    d\rvert_{\f=\tf} J_\f [\Psi] [\tu][\te] = \into \tu \, d\rvert_{\f=\tf} R_\f [\Psi] \, \te^T \abs{\operatorname{det}\operatorname{D}\f} dx + \into \tu \, R_\tf \, \te^T d\rvert_{\f=\tf} \abs{\operatorname{det}\operatorname{D}\f} [\Psi] \, dx.
\end{equation}
Note that 
\begin{equation}\label{rphi}
    \left[ d\rvert_{\f=\tf} R_\f [\Psi] \right] \circ \tfi = - \operatorname{D}\tfi \left[ \operatorname{D}(\Psi \circ \tfi) + (\operatorname{D}(\Psi \circ \tfi))^T \right] (\operatorname{D}\tfi)^T.
\end{equation}
Therefore by equalities \eqref{derdet}, \eqref{derjphi}, \eqref{rphi} and a change of variables, equality \eqref{derjphilemma} follows.

We now compute the second term in the right-hand side of \eqref{diff1}. Obviously, we have that 
$$d\rvert_{\f=\tf} T_{\f}[\Psi] [\tu] [\te] = \into \operatorname{curl}\tu \left( d\rvert_{\f=\tf} G_\f [\Psi]\right) (\operatorname{curl}\te)^T \, dx + \tau \ d\rvert_{\f=\tf} \into \frac{N(\f,\tu,\te)}{\abs{\operatorname{det}(\operatorname{D}\f)}} \, dx,$$
where we have set 
$\displaystyle G_\f  = \frac{(\operatorname{D}\f)^T \operatorname{D}\f}{\abs{\operatorname{det}(\operatorname{D}\f)}}$ and $N(\f,\tu,\te) = N(\f , \tu)N(\f, \te)$  with  
$$
N(\f,\eta)= 
\operatorname{div}_x \left( \eta  (\operatorname{D}\f)^{-1} (\operatorname{D}\f)^{-T} \operatorname{det}(\operatorname{D}\f)  \right) ,
$$
for any vector field $\eta $. 
To shorten  our notation, we also  set $\zeta = \Psi \circ \tfi$. By a change of variables one can see that 
\begin{align*}
&\into \operatorname{curl}\tu \ d\rvert_{\f=\tf} G_\f \ (\operatorname{curl}\te)^T \, dx =\\
&=\into \operatorname{curl}\tu \ \frac{(\operatorname{D}\Psi)^T \operatorname{D}\tf + (\operatorname{D}\tf)^T \operatorname{D}\Psi }{\abs{\operatorname{det}(\operatorname{D}\tf)}} \ (\operatorname{curl}\te)^T \, dx \\
& \qquad -\into \operatorname{curl}\tu \ \frac{(\operatorname{D}\tf)^T \operatorname{D}\tf \ d\rvert_{\f=\tf} \abs{\operatorname{det}(\operatorname{D}\f)}[\Psi]}{\left( \operatorname{det}(\operatorname{D}\tf) \right)^2} \ (\operatorname{curl}\te)^T\, dx \\
&= \int_{\tf(\Omega)} \operatorname{curl}\tv \left( \operatorname{D}\zeta + (\operatorname{D}\zeta)^T \right) (\operatorname{curl}\tw)^T \, dy - \int_{\tf(\Omega)} \operatorname{curl}\tv \, (\operatorname{curl}\tw)^T \operatorname{div}(\zeta) \, dy.
\end{align*}
Now by standard calculus we have 
\begin{eqnarray}\label{denne}
\lefteqn{  d \rvert_{\f=\tf} \int_\Omega \frac{N(\f,\tu,\te)}{\abs{\operatorname{det}(\operatorname{D}\f)}} \, dx =\int_\Omega \frac{d}{dt} \biggl( \frac{N(\tf + t \Psi,\tu,\te)}{|\operatorname{det}(\operatorname{D}(\tf + t \Psi)|} \biggr) \biggr\rvert_{t=0}  dx  }\nonumber \\
    &    &= \int_\Omega \frac{\frac{d}{dt} (N(\tf+t\Psi,\tu,\te) \rvert_{t=0} \  \abs{\operatorname{det}(\operatorname{D}\tf)} - N(\tf,\tu,\te) \, d\rvert_{\f= \tf} \abs{\operatorname{det}(\operatorname{D}\f)} [\Psi]}{(\operatorname{det}(\operatorname{D}\tf))^2} \, dx\nonumber \\
 &  &    = \int_\Omega    \left(     \frac{d}{dt} (N(\tf+t\Psi,\tu) \rvert_{t=0}  N(\tf, \te   ) +N(\tf, \tu   )
  \frac{d}{dt} (N(\tf+t\Psi,\te) \rvert_{t=0}   \right)
 \abs{\operatorname{det}\operatorname{D}\tf}^{-1} \, dx \nonumber\\
& &    -\int_{\tf(\Omega)} \operatorname{div}_y (\tv) \operatorname{div}_y (\tw) \operatorname{div}_y (\zeta) \, dy.
\end{eqnarray} 
Going back to formula \eqref{derjphilemma}, we note that by integrating by parts we obtain that 
\begin{eqnarray*}
  \lefteqn{ 
        -\int_{\tf(\Omega)} \tv \left(\operatorname{D}\zeta + (\operatorname{D}\zeta)^T \right) \tw^T \, dy =-\int_{\tf(\Omega)} \tv_i(\partial_j \zeta_i) \tw_j \, dy - \int_{\tf(\Omega)} \tv_i (\partial_i \zeta_j) \tw_j \, dy } \\
     &  &=\int_{\tf(\Omega)} (\partial_j \tv_i) \tw_j \zeta_i \,dy + \int_{\tf(\Omega)} (\tv \cdot \zeta) \operatorname{div}(\tw) \, dy - \int_{\partial \tf(\Omega)} (\tv \cdot\zeta) (\tw \cdot \nu) \, d\sigma \\
     &  & +\int_{\tf(\Omega)} (\tw \cdot \zeta) \operatorname{div}(\tv) \, dy + \int_{\tf(\Omega)} (\partial_i \tw_j) \tv_i \zeta_j \, dy - \int_{\partial \tf(\Omega)} (\tv \cdot \nu) (\tw \cdot \zeta) \, d\sigma\\
     &  &=\int_{\tf(\Omega)} (\partial_j \tv_i) \tw_j \zeta_i \,dy + \int_{\tf(\Omega)} (\partial_i \tw_j) \tv_i \zeta_j \, dy + \int_{\tf(\Omega)} (\tv \cdot \zeta) \operatorname{div}(\tw) \, dy \\
     &  &  +\int_{\tf(\Omega)} (\tw \cdot \zeta) \operatorname{div}(\tv) \, dy - 2\int_{\partial \tf(\Omega)} (\tv \cdot \tw) (\zeta \cdot \nu) \, d\sigma,
\end{eqnarray*}
where we have used the fact that $\tv \times \nu = 0 = \tw \times \nu$  hence
$$(\tv \cdot \nu) (\tw \cdot \zeta) = (\tv \cdot \zeta) (\tw \cdot \nu) = (\tv \cdot \nu) (\tw \cdot \nu) (\zeta \cdot \nu) = (\tv \cdot \tw) (\zeta \cdot \nu)$$
on $\partial \tf(\Omega)$.
Now, since 
$$\sum_{i,j=1}^3 (\partial_j \tv_i) \tw_j \zeta_i - (\partial_i \tv_j) \tw_j \zeta_i = \operatorname{curl}\tv \cdot (\tw \times \zeta),$$
$$\sum_{i,j=1}^3 (\partial_i \tw_j) \tv_i \zeta_j - (\partial_j \tw_i) \tv_i \zeta_j = \operatorname{curl}\tw \cdot (\tv \times \zeta),$$
we deduce that
\begin{equation*}
    \begin{split}
        &\int_{\tf(\Omega)} (\partial_j \tv_i) \tw_j \zeta_i \,dy + \int_{\tf(\Omega)} (\partial_i \tw_j) \tv_i \zeta_j \, dy \\
        &=\int_{\tf(\Omega)} (\partial_j \tv_i) \tw_j \zeta_i \,dy + \int_{\tf(\Omega)} (\partial_i \tw_j) \tv_i \zeta_j \, dy -\int_{\tf(\Omega)} (\partial_i \tv_j) \tw_j \zeta_i \,dy \\
        &- \int_{\tf(\Omega)} (\partial_j \tw_i) \tv_i \zeta_j \, dy + \int_{\tf(\Omega)} \nabla (\tv \cdot \tw) \cdot \zeta \, dy\\
        &=\int_{\tf(\Omega)} \operatorname{curl}\tv \cdot (\tw \times \zeta) \, dy + \int_{\tf(\Omega)} \operatorname{curl}\tw \cdot (\tv \times \zeta) \, dy + \int_{\tf(\Omega)} \nabla (\tv \cdot \tw) \cdot \zeta \, dy.
    \end{split}
\end{equation*}
Thus we get
\begin{equation*}
    \begin{split}
         &-\int_{\tf(\Omega)} \tv \left(\operatorname{D}\zeta + (\operatorname{D}\zeta)^T \right) \tw^T \, dy \\
         &= \int_{\tf(\Omega)} \tv \cdot (\zeta \times \operatorname{curl}\tw) \, dy + \int_{\tf(\Omega)} \tw \cdot (\zeta \times \operatorname{curl}\tv) \, dy + \int_{\tf(\Omega)} \nabla (\tv \cdot \tw) \cdot \zeta \, dy \\
         & + \int_{\tf(\Omega)} (\tv \cdot \zeta) \operatorname{div}\tw \, dy +\int_{\tf(\Omega)} (\tw \cdot \zeta) \operatorname{div}\tv \, dy -2 \int_{\partial \tf(\Omega)} (\tv \cdot \tw)(\zeta \cdot \nu) \, d\sigma.
    \end{split}
\end{equation*}

Since $v$ and $w$ satisfy the equation in \eqref{mainpenweak} on $\tilde\Phi (\Omega)$, we have
\begin{equation}\label{formulone}
    \begin{split}
        &\int_{\tf(\Omega)} \tv \cdot (\zeta \times \operatorname{curl}\tw) \, dy  \\
        &=\lambda^{-1} \int_{\tf(\Omega)} \operatorname{curl}\operatorname{curl}\tv \cdot (\zeta \times \operatorname{curl}\tw) \, dy - \lambda^{-1} \tau \int_{\tf(\Omega)} \nabla \operatorname{div}\tv \cdot (\zeta \times \operatorname{curl}\tw) \, dy \\
        &= -\lambda^{-1}\int_{\tf(\Omega)} (\operatorname{curl}\tw \cdot \operatorname{curl}\tv) \operatorname{div}(\zeta) \, dy +  \lambda^{-1} \int_{\tf(\Omega)} \operatorname{curl}\tv \operatorname{D}\zeta \, (\operatorname{curl}\tw)^T \, dy  \\
        & - \lambda^{-1} \int_{\tf(\Omega)} \operatorname{curl}\tv \operatorname{D}(\operatorname{curl}\tw) \, \zeta^T \, dy    - \lambda^{-1} \tau \int_{\tf(\Omega)} \nabla \operatorname{div}\tv \cdot (\zeta \times \operatorname{curl}\tw) \, dy  \\
        & + \lambda^{-1} \int_{\partial\tf(\Omega)} (\zeta \cdot \nu) (\operatorname{curl}\tw \cdot \operatorname{curl}\tv) \, d\sigma
           - \lambda^{-1} \int_{\partial\tf(\Omega)} (\zeta \cdot \operatorname{curl}\tv) (\nu \cdot \operatorname{curl}\tw) \, d\sigma    ,
    \end{split}
\end{equation}
where, in order to compute the boundary integrals, we have used the following formula
\begin{eqnarray}
\lefteqn{ (\nu \times \operatorname{curl} \tv) \cdot (\zeta \times  \operatorname{curl} \tw) } \nonumber  \\
& =\xi_{ijk} \nu_j  ( \operatorname{curl} \tv)_k \xi_{ilm} \zeta_l  ( \operatorname{curl} \tw)_m  = ( \delta_{jl}\delta_{km}- \delta_{jm}\delta_{kl})\nu_j\zeta_l  ( \operatorname{curl} \tv)_k( \operatorname{curl} \tw)_m\nonumber    \\  
& =  (\zeta \cdot \nu) (\operatorname{curl} \tw \cdot  \operatorname{curl} \tv)  - (\nu \cdot  \operatorname{curl} \tw) (\zeta \cdot  \operatorname{curl} \tv)  \, .
\end{eqnarray}  

We note that the last boundary term in formula \eqref{formulone} vanishes by Lemma~\ref{fundamental}.
Then
\begin{equation*}
    \begin{split}
        &-\tilde{\lambda}\int_{\tf(\Omega)} \tv \left(\operatorname{D}\zeta + (\operatorname{D}\zeta)^T \right) \tw^T \, dy \\
        &= -2 \int_{\tf(\Omega)} (\operatorname{curl}\tw \cdot \operatorname{curl}\tv) \operatorname{div}\zeta \, dy + \int_{\tf(\Omega)} \operatorname{curl}\tw \left( \operatorname{D}\zeta + (\operatorname{D}\zeta)^T \right) (\operatorname{curl}\tv)^T \, dy \\
        & - \int_{\tf(\Omega)} \nabla (\operatorname{curl}\tw \cdot \operatorname{curl}\tv) \cdot \zeta \, dy - \tau \int_{\tf(\Omega)} \nabla \operatorname{div}\tv \cdot (\zeta \times \operatorname{curl}\tw) \, dy \\
        & -\tau \int_{\tf(\Omega)} \nabla \operatorname{div}\tw \cdot(\zeta \times \operatorname{curl}\tv) \, dy +2 \int_{\partial \tf(\Omega)} (\zeta \cdot \nu) (\operatorname{curl}\tw \cdot \operatorname{curl}\tv) \, d\sigma \\
        & + \tilde{\lambda}\int_{\tf(\Omega)} \nabla (\tv \cdot \tw) \cdot \zeta \, dy + \tilde{\lambda}\int_{\tf(\Omega)} (\tw \cdot \zeta) \operatorname{div}\tv \, dy + \tilde{\lambda}\int_{\tf(\Omega)} (\tv \cdot \zeta) \operatorname{div}\tw \, dy \\
        & -2 \tilde{\lambda} \int_{\partial \tf(\Omega)} (\tv \cdot \tw) (\zeta \cdot \nu) \, d\sigma = - \int_{\tf(\Omega)} (\operatorname{curl}\tw \cdot \operatorname{curl}\tv) \operatorname{div}\zeta \, dy \\
        & + \int_{\tf(\Omega)} \operatorname{curl}\tw \left( \operatorname{D}\zeta + (\operatorname{D}\zeta)^T \right) (\operatorname{curl}\tv)^T \, dy - \tau \int_{\tf(\Omega)} \nabla \operatorname{div}\tv \cdot (\zeta \times \operatorname{curl}\tw) \, dy \\
        & - \tau \int_{\tf(\Omega)} \nabla \operatorname{div}\tw \cdot (\zeta \times \operatorname{curl}\tv) \, dy - \tilde{\lambda}\inttpo (\tv \cdot \tw) \operatorname{div}\zeta \, dy + \tilde{\lambda}\int_{\tf(\Omega)} (\tw \cdot \zeta) \operatorname{div}\tv \, dy \\
        & + \tilde{\lambda}\int_{\tf(\Omega)} (\tv \cdot \zeta) \operatorname{div}\tw \, dy - \tilde{\lambda} \int_{\partial \tf(\Omega)} (\tv \cdot \tw) (\zeta \cdot \nu) \, d\sigma + \inttpos (\operatorname{curl}\tw \cdot \operatorname{curl}\tv) (\zeta \cdot \nu) \, d\sigma.
        \end{split}
\end{equation*}

Hence
\begin{equation*}
    \begin{split}
        &(\tilde{\lambda}+1)^2 \langle d\rvert_{\f=\tf} S_{\Phi}  [\Psi][\tu], \te \rangle_{\tf}\\
        &  = -\tilde{\lambda}\int_{\tf(\Omega)} \tv \left(\operatorname{D}\zeta + (\operatorname{D}\zeta)^T \right) \tw^T \, dy + \tilde{\lambda}\int_{\tf(\Omega)} \tv \tw^T \operatorname{div}(\zeta) \, dy  \\
        & - \int_{\tf(\Omega)} \operatorname{curl}\tv \left( 
        \operatorname{D}\zeta + (\operatorname{D}\zeta)^T \right) (\operatorname{curl}\tw)^T \, dy  + \int_{\tf(\Omega)} \operatorname{curl}\tv (\operatorname{curl}\tw)^T \operatorname{div}(\zeta) \, dy  \\
        &  - \tau \, \ d \rvert_{\f=\tf} \int_\Omega \frac{N(\f,\tu,\te))}{\abs{\operatorname{det}(\operatorname{D}\f)}} \, dx 
        =\int_{\partial \tf(\Omega)} \left( \operatorname{curl}\tv \cdot \operatorname{curl}\tw  - \tilde{\lambda} \ \tv \cdot \tw \right) (\zeta \cdot \nu) \, d\sigma \\
        &  - \tau \int_{\tf(\Omega)} \nabla \operatorname{div}\tv \cdot (\zeta \times \operatorname{curl}\tw) \, dy - \tau  \int_{\tf(\Omega)} \nabla \operatorname{div}\tw \cdot (\zeta \times \operatorname{curl}\tv) \, dy  + \tilde{\lambda} \int_{\tf(\Omega)} (\tw \cdot \zeta) \operatorname{div}(\tv) \, dy\\
        &  + \tilde{\lambda} \int_{\tf(\Omega)} (\tv \cdot \zeta) \operatorname{div}(\tw) \, dy - \tau \, \ d \rvert_{\f=\tf} \int_\Omega \frac{N(\f,\tu,\te)}{\abs{\operatorname{det}(\operatorname{D}\f)}} \, dx.
    \end{split}
\end{equation*}

By the previous equality, \eqref{denne} and by observing that    
$$\frac{\operatorname{div}_x \left[ \tu (x) (\operatorname{D}\tf(x))^{-1} (\operatorname{D}\tf(x))^{-T} \operatorname{det}(\operatorname{D} \tf(x)) \right]}{\operatorname{det} (\operatorname{D} \tf(x))}= \operatorname{div}_y \tv (\tf(x)) = 0,$$ 
$$ \frac{\operatorname{div}_x \left[ \te (x) (\operatorname{D}\tf(x))^{-1} (\operatorname{D}\tf(x))^{-T} \operatorname{det}(\operatorname{D} \tf(x)) \right]}{\operatorname{det} (\operatorname{D} \tf(x))} = \operatorname{div}_y \tw (\tf(x)) =0,$$
for almost all $x \in \Omega$, we conclude. 
\end{proof}

In the case of 
 domain perturbations depending real analytically on one scalar parameter,  it is possible to prove a Rellich-Nagy-type theorem and describe 
 all the eigenvalues
splitting from a multiple eigenvalue of multiplicity $m$ by means of $m$ real-analytic functions.

\begin{thm}\label{nagy}Let $\Omega$ be a bounded domain in ${\mathbb{R}}^N$ of class ${\mathcal{C}}^{1,1}$.  Let  $\tilde \Phi \in {\mathcal{A}}_{\Omega}$ and $\{\Phi_{\epsilon}\}_{\epsilon\in {\mathbb{R}}} \subset {\mathcal{A}}_{\Omega}$ be a family depending real-analytically on $\epsilon$ such that $\Phi_0=\tilde \Phi$. Assume that  $\tilde\lambda$ is a Maxwell eigenvalue on $\tilde\Phi (\Omega)$ of multiplicity $m$,  $\tilde \lambda =\lambda_{n}[\tilde \Phi ]=\dots =\lambda_{n+m-1}[\tilde \Phi]$ for some $n\in {\mathbb{N}}$ and that      $\tilde{E}^{(1)},\dots, \tilde{E}^{(m)  }\in \xn (\di \,0, \tilde\Phi (\Omega))$  is an orthonormal basis of the eigenspace of $\tilde \lambda $, the orthonormality being taken in $(L^2(\tf (\Omega ))^3$.  Moreover, assume that $\tilde{E}^{(i)} \in  H^2(\tilde\Phi(\Omega))$ for all $i=1,\dots , m$.

Then there exists an open interval $I$ containing zero and $m$ real-analytic functions $g_{1},\dots , g_m$ from $I$ to ${\mathbb{R}}$ such that $\{\lambda_{n}[ \Phi_{\epsilon} ],\dots ,\lambda_{n+m-1}[\Phi_{\epsilon}]\}=\{g_1(\epsilon ), \dots  , g_m(\epsilon) \}$ for all $\epsilon \in I$. Moreover, the derivatives 
 $g'_1(0), \dots , g_m'(0)$ at zero of the functions $g_1, \dots , g_m$ coincide with the eigenvalues of the matrix 
\begin{equation}\label{nagymatrix}\left( \inttpos \left( \lambda_F [\tf]      \tilde{E}^{(i)}\cdot \tilde{E}^{(j)}   - \operatorname{curl} \tilde{E}^{(i)}\cdot  \operatorname{curl} \tilde{E}^{(j)}\right)  \zeta \cdot \nu  \, d\sigma  \right)_{i,j=1,\dots ,m}
\end{equation}
where  $\zeta = {\dot{\Phi } }_0\circ \tfi   $,    ${\dot{\Phi } }_0$ denotes the derivative at zero of the map $\epsilon \mapsto \Phi_{\epsilon} $.
\end{thm}

\begin{proof} First of all, we note that by  our assumptions, $\tilde \lambda$ does not coincide with any of the eigenvalues of the Dirichlet Laplacian in $\Omega$ multiplied by $\tau$, see Remark~\ref{risonanzarem}, and  by the well-known continuity of the eigenvalues of the Dirichlet Laplacian, this implies that for all $\epsilon $ in a  sufficiently small neighbourhood of zero, the eigenvalues    $\{\lambda_{n}[ \Phi_{\epsilon} ],\dots ,\lambda_{n+m-1}[\Phi_{\epsilon}]\}$   satisfy the same property.  By applying  \cite[Theorem~2.27, Corollary~2.28]{lala}  to the family of operators   $S_{\Phi_{\epsilon} }$
we deduce that there exists an open interval $I$ containing zero and $m$ real-analytic functions $\tilde g_{1},\dots , \tilde g_m$ from $I$ to ${\mathbb{R}}$ such that $\{  (1+ \lambda_{n}[ \Phi_{\epsilon} ])^{-1},\dots ,(1+\lambda_{n+m-1}[\Phi_{\epsilon}] )^{-1} \}=\{\tilde g_1(\epsilon ), \dots  ,\tilde  g_m(\epsilon) \}$ for all $\epsilon \in I$; moreover, the derivatives    of those functions at zero are given by the eigenvalues of the matrix
$$  \left( \langle d\rvert_{\f=\tf} S_{\Phi}  [\Psi][\tu^{(i)}], \tu^{(j)} \rangle_{\tf}\right)_{i,j=1,\dots , m} \, , $$
where $\tu^{(i)}=\tilde{E}^{(i)}\circ \tilde\Phi$ for all $i=1,\dots ,m$. The proof follows by setting $g_i(\epsilon)= \tilde g_i^{-1}(\epsilon )-1$, $i=1, \dots , m$, and using  Lemma~\ref{techlemma}. 
\end{proof}

We conclude this section by proving an immediate consequence of our results, namely the  Rellich-Pohozaev identity for Maxwell eigenvalues. 

\begin{thm}[Rellich-Pohozaev Identity] Let $\Omega$ be a bounded domain in ${\mathbb{R}}^3$ of class ${\mathcal{C}}^{1,1}$. Let $\lambda $ be a Maxwell eigenvalue and $E\in \xn (\di, 0, \Omega ) $ a corresponding nontrivial eigenvector normalized in $(L^2(\Omega ))^3 $.  Assume that $E\in H^2(\Omega)$. Then
$$
\lambda =\frac{1}{2} \int_{\partial \Omega} \left(  \abs{\operatorname{curl} E}^2  -\lambda  \abs{E}^2 \right) \left( x \cdot \nu \right) \, d\sigma \, .
$$
\end{thm}

\begin{proof}
Assume that $\lambda =\lambda_n(\Omega )=\dots =\lambda_{n+m-1}(\Omega)$ is a Maxwell eigenvalue with multiplicity $m$ (with the understanding that the corresponding $m$-dimensional eigenspace is made of Maxwell eigenvectors, see Remark~\ref{risonanzarem}).  
 We consider a family of dilations  $(1+\epsilon )\Omega$  of $\Omega$ which can viewed as a family of  diffeomorphisms  $\Phi_{\epsilon }=I+\epsilon I$, $\epsilon \in {\mathbb{R}}$.   It is obvious that   $\lambda _{n+i}[\Phi_{\epsilon } ]= \left( 1+\epsilon \right) ^{-2}\lambda  $ for all $i=0, \dots , m-1$.
 In particular, the domain perturbation under consideration preserves the multiplicity of $\lambda $ and the matrix \eqref{nagymatrix} is a multiple of the identity. 
By  differentiating with respect to $\epsilon $ and applying Theorem~\ref{nagy} with $\Phi = \zeta =I$,  we obtain
\begin{equation}\label{rellich1}
\left. \frac{d}{d\epsilon }\lambda _{n}[\Phi _{\epsilon}]\right\vert
_{\epsilon=0}=     \int_{\partial \Omega} \left( \lambda  \abs{E}^2  - \abs{\operatorname{curl} E}^2      \right) \left( x \cdot \nu \right) \, d\sigma ,
\end{equation}
where the given normalized eigenvector $E$ is serving as an element of the orthonormal basis of the eigenspace. 
If we differentiate the equality $\lambda _{n}[\Phi_{\epsilon } ]= \left( 1+\epsilon \right) ^{-2}\lambda  $   with respect to $\epsilon$, we obtain
\begin{equation}\label{rellich2}
\left. \frac{d}{d\epsilon }\lambda _{n}[\Phi _{\epsilon}]\right\vert_{\epsilon=0}=\left. \frac{d}{dt}\left( \left( 1+\epsilon \right) ^{-2}\lambda   \right)\right\vert _{t=0}=-2\lambda .
\end{equation}

By  combining \eqref{rellich1} and  \eqref{rellich2} we conclude.
\end{proof}

\section{Criticality for symmetric functions of the eigenvalues}
\label{oversec}

We denote by  $\mathcal{V}[\f]$ the measure of $\f(\Omega)$, that is 
\begin{equation}\label{eqvolume}
\mathcal{V}[\f] \colon = \int_{\f(\Omega)}dx = \into \abs{\operatorname{det}\operatorname{D}\f} \, dx,
\end{equation}
and by   $\mathcal {P}[\Phi]$ the perimeter of $\f(\Omega)$ that is 
\begin{equation}\label{eqperimeter}
\mathcal {P}[\Phi ]:= \int_{\partial \f(\Omega)}d\sigma  =\int_{\partial \Omega } \left|\nu (\operatorname{D} \Phi )^{-1}   \right| \left|\mathrm{det}\operatorname{D} \Phi\right|d\sigma \, .
\end{equation}

We are interested in extremum problems of the type
\begin{equation} \label{extrproblems}
\min_{\mathcal{V}[\f]={\rm const.}} \Lambda_{F,s}[\f]\  \text{ or } \max_{\mathcal{V}[\f]={\rm const.}} \Lambda_{F,s}[\f]  ,\end{equation}
as well as problems of the type
\begin{equation} \label{extrproblemsper}
\min_{\mathcal {P}[\f]={\rm const.}} \Lambda_{F,s}[\f]\  \text{ or } \max_{\mathcal{P}[\f]={\rm const.}} \Lambda_{F,s}[\f]  .\end{equation}

It is convenient to recall the following lemma from \cite{la2014}, where ${\mathcal{H}}= \di \nu$ denotes the mean curvature of $\Omega $, that is, the sum of the principal curvatures.

\begin{lem}
Let $\Omega$ be a bounded domain in $\mathbb{R}^N$ of class ${\mathcal{C}}^{1,1}$. Then the maps $\mathcal{V}$, $\mathcal{P}$ from $\mathcal{A}_\Omega$ to $\mathbb{R}$ defined in \eqref{eqvolume} are real-analytic. Moreover, the differentials of $\mathcal{V}$ and $\mathcal{P}$ at any point $\tf \in \mathcal{A}_\Omega$ are  given by the formulas 
    \begin{equation}\label{diffvol}
        d\rvert_{\f=\tf} \mathcal{V}[\f] [\psi] = \inttpos  (\Psi \circ \tfi (x)) \cdot \nu (x)  \, d\sigma , 
    \end{equation}
    and 
    \begin{equation}\label{diffper}
     d\rvert_{\f=\tf} \mathcal{P}[\f] [\psi]   = \int_{\partial\tilde  \Phi (\Omega )}{\mathcal H}(x)\,  (\Psi \circ \tfi  (x)) \cdot \nu(x)   d\sigma ,
    \end{equation}
    for all $\Psi \in \mathcal{C}^{1,1}(\overline{\Omega}, \mathbb{R}^3)$.
\end{lem}

 For $\alpha \in ]0, +\infty[$, we set 
    $$V(\alpha  ) \colon = \set{\f \in \mathcal{A}_\Omega : \mathcal{V}(\f) = \alpha  },\ \ P(\alpha ) \colon = \set{\f \in \mathcal{A}_\Omega : \mathcal{P}(\f) =\alpha   }\, .$$
   
Keeping in mind the Lagrange Multiplier Theorem (which holds also in infinite dimensional spaces), we note that if $\tilde \Phi \in \mathcal{A}_\Omega [F]$
is a minimizer/maximizer in \eqref{extrproblems} or \eqref{extrproblemsper} respectively, then it is a critical point for the function $\f \mapsto  \Lambda_{F,s}[\f]$ under the constraint $\f \in V(\tilde \alpha   )$ or the constraint   $\f \in P(\tilde \beta   )$ respectively, where $\tilde \alpha =  {\mathcal{V}}(\tilde\Phi    )$  and $\tilde \beta = {\mathcal{P}}(\tilde\Phi    ) $, which means  that 
\begin{equation}\label{kerker}
{\rm Ker }\, d\rvert_{\f=\tf} \mathcal{V}[\f] \subset {\rm Ker }\, d\rvert_{\f=\tf}  \Lambda_{F,s} [\f] ,
\end{equation}
or
\begin{equation}\label{kerkerper}
{\rm Ker }\, d\rvert_{\f=\tf} \mathcal{P}[\f] \subset {\rm Ker }\, d\rvert_{\f=\tf}  \Lambda_{F,s} [\f] ,
\end{equation}
respectively. 

The following theorem provides a characterization of those points $\tilde \Phi \in  \mathcal{A}_\Omega [F]$ satisfying \eqref{kerker} or \eqref{kerkerper}.

\begin{thm}\label{over}  Let $\Omega$ be a bounded domain in $\mathbb{R}^N$ of class ${\mathcal{C}}^{1,1}$.
Let  F be a non-empty finite subset of $\mathbb{N}$ and $\tilde \alpha  \in ]0,+\infty[$. The following statements hold:
\begin{itemize}
\item[(i)]
Assume that  $\tf \in  \Theta_{\Omega}[F]\cap V( \tilde\alpha )$ is such that $\lambda_j[\tf]$  are  Maxwell eigenvalues with  common value  $\lambda_F[\tf]$ for all $j \in F$. Assume that $\tilde{E}^{(1)},\dots, \tilde{E}^{(\abs{F})}\in  X_N(\di\, 0, \tf(\Omega))  $ is an orthonormal basis of the eigenspace corresponding to  $\lambda_F[\tf]$ and that those eigenvectors belong to $H^2(\tilde\Phi (\Omega))$.

For $s=1, \dots, \abs{F}$ the function $\tf$ is a critical point for $\Lambda_{F,s}$ with volume constraint $\f \in V(\tilde\alpha ) $  (that is, condition \eqref{kerker} is satisfied)  if and only if there exists a constant $c \in \mathbb{R}$ such that 
\begin{equation}\label{condcrit}
   \sum_{l=1}^{\abs{F}} \left( \lambda_F [\tf] \abs{\tilde{E}^{(l)}}^2 - \abs{\operatorname{curl} \tilde{E}^{(l)}}^2 \right) = c, \qquad {\rm on}\ \partial\tf(\Omega).
\end{equation}
\item[(ii)]
Assume that $\tf \in  \Theta_{\Omega}[F]\cap P( \tilde \alpha  )$ is such that $\lambda_j[\tf]$  are  Maxwell eigenvalues with  common value  $\lambda_F[\tf]$ for all $j \in F$. Assume  that $\tilde{E}^{(1)},\dots, \tilde{E}^{(\abs{F})}  \in  X_N(\di\, 0, \tf(\Omega))$    is an orthonormal basis of the eigenspace corresponding to  $\lambda_F[\tf]$ and that those eigenvectors belong to $H^2(\tilde\Phi (\Omega))$.
  For $s=1, \dots, \abs{F}$ the function $\tf$ is a critical point for $\Lambda_{F,s}$ with perimeter constraint $\f \in P( \tilde \alpha)$  (that is, condition \eqref{kerkerper} is satisfied)  if and only if there exists  a constant $c \in \mathbb{R}$ such that 
\begin{equation}\label{condcritper}
   \sum_{l=1}^{\abs{F}} \left( \lambda_F [\tf] \abs{\tilde{E}^{(l)}}^2 - \abs{\operatorname{curl} \tilde{E}^{(l)}}^2 \right) = c {\mathcal{H}}, \qquad {\rm on}\ \partial\tf(\Omega).
\end{equation}
\end{itemize}
\end{thm}
\begin{proof}
It suffices to observe that by standard linear algebra condition \eqref{kerker} or condition \eqref{kerkerper} is satisfied if and only if there exists a constant $l\in {\mathbb{R}}$ such that 
$ d\rvert_{\f=\tf}  \Lambda_{F,s} [\f]=  ld\rvert_{\f=\tf} \mathcal{V}[\f]$ or $ d\rvert_{\f=\tf}  \Lambda_{F,s} [\f]=  ld\rvert_{\f=\tf} \mathcal{P}[\f]$ respectively . By formulas \eqref{hadamard}, \eqref{diffvol},  \eqref{diffper}  and the Fundamental Lemma of the Calculus of Variations, 
we conclude. 
\end{proof}

In the next theorem, we show that balls are critical domains.

\begin{thm}\label{balls}
Let $\Omega$ be a bounded domain in ${\mathbb{R}}^3$ of class ${\mathcal{C}}^{1,1}$.  Let $\tf\in {\mathcal {A}}_{\Omega}$ be such that $\tf(\Omega)$ is a ball. Let $\tilde{\lambda}$ be a Maxwell eigenvalue  in $\tf(\Omega)$ with an eigenspace of dimension $m$ in   $\xn (\di\, 0, \Tilde \Phi (\Omega))$. Assume that $\lambda_{n-1}[\tilde\Phi (\Omega)]< \lambda_{n}[\tilde\Phi (\Omega)]= \dots =\lambda_{n+m-1}[\tilde\Phi (\Omega)]<\lambda_{n+m}[\tilde\Phi (\Omega)]$ for some $n\in \N$, and let $F=\{n,\dots , n+m-1\}$. Then $\tf$ is both a critical point for    $\Lambda_{F,s} $ with volume constraint $\f\in  V(\mathcal{V}(\tf))$ and  a critical point for    $\Lambda_{F,s} $ with perimeter constraint $\f\in  P(\mathcal{P}(\tf))$, for all $s=1,\dots, \abs{F}$.
\end{thm}
\begin{proof}
Conditions \eqref{condcrit} and  \eqref{condcritper}   are satisfied thanks to  Lemma~\ref{radial} below.
\end{proof} 
\begin{lem}\label{radial}
Let $B$ be a ball in $\mathbb{R}^3$ centered at zero. Let $\lambda$ be a Maxwell eigenvalue  in $B$  with an eigenspace of dimension $m$ in   $\xn (\di\, 0, B)$ and  let $E^{(1)},\dots,E^{(m) }$ be a corresponding  orthonormal basis. Then, the functions 
\begin{equation*}
    \sum_{l=1}^{m  } \abs{E^{(l)}}^2, \ \ \sum_{l=1}^{m  } \abs{\operatorname{curl} E^{(l)}}^2
\end{equation*}
are radial.
\end{lem}
\begin{proof}
Let $E$ be an eigenvector of problem \eqref{main} with eigenvalue $\lambda$. Take an orthogonal matrix $A \in O_3(\mathbb{R})$ and consider the vector field $u$ defined by  $ u  = (E \circ A) A$. 
Then
\begin{equation*}
    \operatorname{D}u (x)= A^T \operatorname{D}E (Ax) \, A.
\end{equation*}
Note that in this proof, for simplicity,  we denote by   $Ax$  the row vector   $(Ax^T)^T$ which is identified with the image of $x$ via the linear transformation associated with the matrix $A$.
Thus
\begin{equation*}
    \operatorname{div}u~(x) = \operatorname{Tr}(\operatorname{D}u (x)) = \operatorname{Tr}(A^T \operatorname{D}E (Ax) \, A) = \operatorname{Tr}(\operatorname{D}E(x)) = \operatorname{div}E~ (Ax)\, .
\end{equation*}
Moreover, we have
\begin{align*}
 &   \Delta u_i (x)  = \frac{\partial}{\partial x_k} \frac{\partial u_i}{\partial x_k}(x) =\frac{\partial}{\partial x_k}[A_{rk} (\partial_r E_j) (Ax) A_{ji}]  \\
    &= A_{rk} (\partial_s \partial_r E_j) (Ax) A_{sk} A_{ji}= (\partial_s \partial_r E_j)(Ax) \delta_{rs} A_{ji} = \partial^2_r E_j (Ax) A_{ji} = \left[ \Delta E (Ax) \, A \right]_i\, .
\end{align*}
Therefore the vector laplacian of $u$ satisfies
\begin{equation*}
    \Delta u (x) = \Delta E (Ax) \, A  \, .
\end{equation*}
Finally, we get that 
\begin{align*}
    \operatorname{curl}\operatorname{curl} u (x) &=   \operatorname{D}  \operatorname{div} u (x) - \Delta u (x) = \left[ (\operatorname{D} \operatorname{div}  E)(Ax) - (\Delta E) (Ax)\right] A \\
    &= \left[ (\operatorname{curl}\operatorname{curl}E) (Ax) \right] A = \lambda E (Ax) A = \lambda u (x).
\end{align*}
This proves that if $E^{(1)},\dots, E^{(m ) }$ is an orthonormal basis of the eigenspace associated with the eigenvalue $\lambda$, then $\set{u^{(j)}=(E^{(j)} \circ A) A : j=1,\dots, m }$ is another orthonormal basis for the eigenspace associate with $\lambda$. Since both $\set{E^{(j)} : j=1,\dots, m  }$ and $\set{u^{(j)} : j=1,\dots, m }$ are orthonormal bases, then there exists $R[A] \in O_{ m  }(\mathbb{R})$ with matrix $(R_{ij}[A])_{}i,j=1,\dots,\, m$ such that 
\begin{equation*}
    u^{(j)}= \sum_{l=1}^{m } R_{jl}[A] E^{(l)}.
\end{equation*}
Therefore 
\begin{align*}
    \sum_{j=1}^{m} \abs{E^{(j)}}^2 \circ A &= \sum_{j=1}^{m} \abs{\left( E^{(j)} \circ A \right) A}^2 = \sum_{j=1}^{m} \abs{u^{(j)}}^2 \\
    &  = \sum_{j=1}^{m} \left( \sum_{l=1}^{m}  R_{jl}[A] E^{(l)} \right) \cdot \left( \sum_{h=1}^{m}  R_{jh}[A] E^{(h)} \right)\\
    &= \sum_{j=1}^{m} \sum_{l,h=1}^{m} R_{jl}[A] R_{jh}[A] (E^{(l)} \cdot E^{(h)}) = \sum_{l=1}^{m} \abs{E^{(l)}}^2,
\end{align*}
which proves that   $ \sum_{l=1}^{m} \abs{E^{(l)}}^2$ is a radial function. 

Note that  $\displaystyle \operatorname{curl}u^{(j)} = \sum_{l=1}^{m} R_{jl}[A] \operatorname{curl}E^{(l)}.$
By formula \eqref{changecurl0}        we have that
\begin{equation*}
\operatorname{curl}E \circ A = \frac{\operatorname{curl}u\,  A^{T}    }   {\operatorname{det} A}.
\end{equation*}
Thus
\begin{align*}
    \sum_{j=1}^{m} \abs{\operatorname{curl}E^{(j)}}^2 \circ A &= \sum_{j=1}^{m} \left( \operatorname{curl}E^{(j)} \circ A \right) \cdot \left( \operatorname{curl}E^{(j)} \circ A \right)  \\  &  = \sum_{j=1}^{m} \left( \operatorname{curl}u^{(j)} \, A^T \right) \cdot \left( \operatorname{curl}u^{(j)} \, A^T \right) \frac{1}{\operatorname{det}(A)^2} = \sum_{j=1}^{m} \abs{\operatorname{curl}u^{(j)}}^2 \\
    &= \sum_{j=1}^{m} \left( \sum_{l=1}^{m} R_{jl}[A] \operatorname{curl}E^{(l)} \right) \cdot \left( \sum_{h=1}^{m} R_{jh}[A] \operatorname{curl}E^{(h)} \right)\\
     &  = \sum_{l=1}^{m} \delta_{lh} \operatorname{curl}E^{(l)} \cdot \operatorname{curl}E^{(h)} = \sum_{l=1}^{m} \abs{\operatorname{curl}E^{(l)}}^2,
\end{align*}
which proves that $\sum_{l=1}^{m}|  \operatorname{curl} E^{(l)}  |^2$ is also a radial function. 
\end{proof}

\section{Appendix: eigenfunctions on the ball}

Let  $B$ the ball in $\mathbb{R}^3$ of radius $R$ centred at zero. Here we use the spherical coordinates $(\rho, \theta, \varphi) \in [0,R] \times [0,\pi] \times [0,2\pi]$  where $\theta  $ is the polar angle  ($\theta =0$ at the north pole) and  $\varphi$ is the azimuthal angle.  It is also convenient to use the standard local orthonormal base
$(    \mbox{\boldmath$ \hat\rho$} ,   \mbox{\boldmath$ \hat\theta$},   \mbox{\boldmath$ \hat\varphi$}   )$  canonically associated with  $(\rho, \theta , \varphi)$, namely
$$
\begin{array} {l}
   \mbox{\boldmath$ \hat\rho$}=(\sin \theta \cos\varphi, \sin \theta \sin\varphi , \cos \theta)\\
   \mbox{\boldmath$ \hat\theta$}=(\cos \theta \cos\varphi, \cos\theta \sin\varphi, -\sin\theta)\\
     \mbox{\boldmath$ \hat\varphi$}=(  - \sin\varphi,  \cos\varphi , 0)
\end{array}
$$
The eigenpairs $(\lambda , u)$ of problem \eqref{main} in $B$ can be expressed in terms of  the Riccati-Bessel functions $\psi_n$ and the spherical harmonics $Y_n^m$.

Recall that $\psi_n(z) = \sqrt{\frac{\pi z}{2}} J_{n+\frac{1}{2}}(z)$, where $J_{n+\frac{1}{2}}$ denote Bessel functions of the first kind and  half-integer order, and that 
the functions $\psi_n$ satisfy the differential equation
$$z^2 f''(z) + (z^2 - n(n+1))f(z) = 0.$$
Recall also that the spherical harmonics $Y^m_n(\theta, \varphi)$, with $|m|\le n$, are eigenfunctions of the Laplace-Beltrami operator $\Delta_{\mathbb{S}^2}$ on the unit sphere  $\mathbb{S}^2$, namely 
$$\Delta_{\mathbb{S}^2} Y^m_n + n(n+1) Y^m_n =0.$$
For more details on these functions we refer to \cite{abste}.

It it  proved in \cite{coda2} and in \cite{hanyak}  that 
the eigenpairs $(\lambda , u)$ of problem \eqref{main} in $B$ are given by the union  of the two families
\begin{equation}\label{costaball}\left\{k^2_{nh}, \operatorname{curl}[Y^m_n (\theta, \varphi) \psi_n(k_{nh} \rho)   \mbox{\boldmath$ \hat\rho$}     ] \right\}_{nmh}\, {\rm and }\  \left\{(k'_{nl})^2, \operatorname{curl}\operatorname{curl}[Y^m_n (\theta, \varphi) \psi_n(k'_{nl} \rho)   \mbox{\boldmath$ \hat\rho$} ]    \right\}_{nml}
\end{equation}
where $n, h, l \in {\mathbb{N}}$,  $m\in \mathbb{Z}$ with $|m|\le n$.
Here  $k_{nh}$, $h\in \mathbb{N}$ denote  the positive zeros of the function $k \mapsto \psi_n(kR)$, arranged in increasing order and  $k'_{nl}$, $l\in {\mathbb{N}}$ denote the  positive zeros of the function $k \mapsto \psi'_n(kR)$, arranged  in the same way.

Now, we compute explicitly the eigenvectors in \eqref{costaball}. 
Recalling the formula   $  \operatorname{curl}(q  \hat{\rho})=\nabla q \times \hat{\rho}$ we have 
\begin{eqnarray}\label{expl1}  \lefteqn{
 \operatorname{curl}[Y^m_n (\theta, \varphi) \psi_n(k_{nh} \rho) \hat{\rho}]  }\nonumber \\
 & & 
 =  
\frac{1}{\rho}\left(
	\frac{1}{\operatorname{sin}\theta} \partial_\varphi \left(Y^m_n(\theta, \varphi)\right) \psi_n(k_{nh} \rho)      \mbox{\boldmath$ \hat\theta$}     
	-\partial_\theta \left(Y^m_n(\theta, \varphi)\right) \psi_n(k_{nh} \rho)     \mbox{\boldmath$ \hat\varphi$}    
	\right)  \, .
\end{eqnarray}
Note that the vector in \eqref{expl1} is zero  if and only if $n=0$.
Similarly, using the formula $\operatorname{curl}\operatorname{curl}( q \hat{\rho}) = \nabla(\partial_\rho q) - \rho \Delta(\frac{q}{\rho}) \hat{\rho}$ and the fact that  $-\rho \Delta(Y^m_n   (\theta, \varphi )\psi_n (k\, \rho )/\rho)=k^2Y^m_n (\theta , \varphi) \psi_n(k\, \rho )  $, see \cite{coda2}, 
 we have that
\begin{eqnarray}\label{expl2}  \lefteqn{
 \operatorname{curl}\operatorname{curl}[Y^m_n (\theta, \varphi) \psi_n(k'_{nl} \rho) \hat{\rho}]  }\nonumber  \\
& = 
(k'_{nl})^2 Y^m_n(\theta, \varphi) \left[ \psi''_n(k'_{nl} \rho) + \psi_n(k'_{nl} \rho) \right]   \mbox{\boldmath$ \hat\rho$}  +
\frac{1}{\rho} k'_{nl} \partial_\theta \left(Y^m_n(\theta, \varphi) \right) \psi'_n(k'_{nl} \rho)  \mbox{\boldmath$ \hat\theta$} \nonumber \\
& +   \frac{1}{\rho \operatorname{sin}\theta} k'_{nl} \partial_\varphi \left(Y^m_n(\theta, \varphi) \right) \psi'_n(k'_{nl} \rho)  \mbox{\boldmath$ \hat\varphi$} 
=\frac{1}{\rho}   \left(
   \frac{n(n+1)}{\rho} Y^m_n(\theta, \varphi) \psi_n(k'_{nl} \rho)    \mbox{\boldmath$ \hat\rho$}  \right. \nonumber\\
& \left. k'_{nl} \partial_\theta \left(Y^m_n(\theta, \varphi) \right) \psi'_n(k'_{nl} \rho)  \mbox{\boldmath$ \hat\theta$} +
\frac{1}{\operatorname{sin}\theta} k'_{nl} \partial_\varphi \left(Y^m_n(\theta, \varphi) \right) \psi'_n(k'_{nl} \rho)  \mbox{\boldmath$ \hat\varphi$}
\right), 
\end{eqnarray}
and again this vector is zero if and only if $n=0$.

Then, we are ready to prove the following theorem. Recall that the Riccati-Bessel function $\psi_1$ is given by   $\psi_1(z)=\frac{\operatorname{sin}z}{z}-\operatorname{cos}z$.

\begin{thm}\label{firsteigen} The first Maxwell eigenvalue in a ball of radius $R$ centred at zero is $(k'_{11})^2$ where $k_{11}'$ is the first positive zero of the derivative of the rescaled Riccati-Bessel function $k\mapsto \psi'_1(Rk)$. Its multiplicity is three and the corresponding Electric eigenspace is generated by the three Electric fields
\begin{eqnarray} E^{(m)}(\rho, \theta, \varphi ) = \lefteqn{\frac{1}{\rho}\left(
\frac{2}{\rho} Y^m_1(\theta, \varphi) \psi_1(k'_{11} \rho)    \mbox{\boldmath$ \hat\rho$}     +
k'_{11} \partial_\theta \left(Y^m_1(\theta, \varphi) \right) \psi'_1(k'_{11} \rho)     \mbox{\boldmath$ \hat\theta $} \right. }\nonumber \\
& &\qquad\qquad\qquad+ \left.
\frac{1}{\operatorname{sin}\theta} k'_{11} \partial_\varphi \left(Y^m_1(\theta, \varphi) \right) \psi'_1(k'_{11} \rho)     \mbox{\boldmath$ \hat\varphi$}
\right),
\end{eqnarray}
for $m=-1,0,1$. The associated magnetic fields are given by
\begin{eqnarray}  \lefteqn{
H^{(m)}(\rho, \theta, \varphi) =-\frac{{\rm i} }{k_{11}'  }  \cu  E^{(m)}(\rho, \theta, \varphi)  } \nonumber  \\  & & 
   =   \frac{{\rm i} k'_{11}}{\rho}    \left(
	\frac{1}{\operatorname{sin}\theta} \partial_\varphi \left(Y^m_1(\theta, \varphi)\right) \psi_1(k'_{11} \rho)   \mbox{\boldmath$ \hat\theta $} 
	-\partial_\theta \left(Y^m_1(\theta, \varphi)\right) \psi_1(k'_{11} \rho)   \mbox{\boldmath$ \hat\varphi$}  
	\right)\, .
\end{eqnarray}
\end{thm}

\begin{proof}
Recall that $\psi_n(z)= z j_n(z)$ where $j_n$ are the spherical Bessel functions of the first kind.
Due to the above observations, we  need to find the smallest positive number $\bar{z}>0$ such that there exists $n \geq 1$ with either $\psi_n(\bar{z})=0$ or $\psi'_n(\bar{z})$; the first eigenvalue would then be $(\bar{z}/R)^2$. Notice that the positive zeros of $\psi_n$ coincide with the zeros of $j_n$. 
First, we recall a useful result about the zeros of the spherical Bessel functions and their derivatives. Denote by $a_{n,s}$ and by $a'_{n,s}$ the $s$-th positive zero of the function $j_n$ and $j'_n$ respectively, for all $n \in \N$; then we have the following interlacing relations:
\begin{equation*}
a_{n,1} < a_{n+1,1} < a_{n,2} < a_{n+1,2} < a_{n,3} < \cdots
\end{equation*}
and 
\begin{equation*}
a'_{n,1} < a'_{n+1,1} < a'_{n,2} < a'_{n+1,2} < a'_{n,3} < \cdots.
\end{equation*}
For a proof of these relations we refer to \cite{liuzou}.
From this we can easily deduce that for each $s\in \mathbb{N}$, the sequences $\{a_{n,s}\}_{n=1}^\infty$ and $\{a'_{n,s}\}_{n=1}^\infty$ are strictly monotonically increasing.

Observe that since the functions $\psi_n$ are smooth and $\psi_n(0)=0$ for all $n \in \N$, the number $\bar{z}$ we are looking for is the first positive zero of $\psi'_n$ for some $n \in \N$. We claim that it is the first positive zero of the function 
$$\psi'_1(z)= \frac{\operatorname{cos}z}{z}+\operatorname{sin}z - \frac{\operatorname{sin}z}{z^2},$$
i.e. $\bar{z}\sim 2.74 \pm 0.01$. To prove this, note that  
\begin{equation}\label{derricbes}
\psi'_n(z) = j_n(z) +z j'_n(z).
\end{equation}
Since 
$$j_n(z) = z^n \sum_{m=0}^\infty \frac{(-1)^m}{m! (2m +2n +1){!}{!}} \left(\frac{z^2}{2}\right)^m,$$
then $j_n(0)=0$ and $j_n(z)>0$ for all $n\in \N$ and for all $z$ between zero and $a_{n,1}$.

Then by \eqref{derricbes}, $\psi'_n (z)>0$ for all $z\in ]0, a'_{n,1}[$.  Due to the monotonicity of the sequence $a'_{n,1}$, $n\in \N$, it is then sufficient to prove that $\bar z \in ]0, a'_{2,1}[$, because in this way the first positive zero of all other functions $\psi'_n, n\geq 2$, will be necessarily larger. 
    Since  $a'_{2,1}\sim 3.34 \pm 0.01$, the claim is proved, and the first eigenvalue  is $k'_{11}=(\bar{z}/R)^2$, where $\bar{z}$ is the first positive zero of the function $\psi'_1$.  The eigenvectors and their curls are computed by using formulas \eqref{expl1} and \eqref{expl2} above and it is easily seen that the  multiplicity is three. 

\end{proof}

{\bf Acknowledgments:} The authors are  very thankful to   Prof. Giovanni S. Alberti for pointing out reference  \cite{alberti} and to 
 Prof. Shuichi Jimbo for  helping them in identifying formula (4-88) in  the book \cite{hira} (in Japanese). 
They are also very thankful  to Prof. Ioannis G.  Stratis for  useful discussions and references concerning the mathematical theory of electromagnetism, 
and  to Prof. Enrique Zuazua for bringing to their attention the method which allows to deduce Rellich-type identities from  Hadamard-type formulas.  
The author acknowledge financial support from the research project BIRD191739/19  ``Sensitivity analysis of partial differential equations in the mathematical theory of electromagnetism'' of the University of Padova. 
The authors are members of the Gruppo Nazionale per l'Analisi Matematica, la Probabilit\`a e le loro Applicazioni (GNAMPA) of the Istituto Nazionale di Alta Matematica (INdAM).

\vspace{36pt}

\noindent Pier Domenico Lamberti\\
Dipartimento di Matematica `Tullio Levi-Civita'\\
University of Padova\\
 Via Trieste 63\\
35121 Padova\\ 
Italy\\
 e-mail: lamberti@math.unipd.it\\

\noindent  Michele Zaccaron\\
Dipartimento di Matematica `Tullio Levi-Civita'\\
University of Padova\\
 Via Trieste 63\\
35121 Padova\\ 
Italy\\
 e-mail:  zaccaron@math.unipd.it\\

\end{document}